\documentclass[12pt]{amsart}
\usepackage[utf8]{inputenc}
\usepackage[lite, alphabetic, nobysame]{amsrefs}
\usepackage{amsmath,mathtools}
\usepackage{amssymb}
\usepackage{srcltx}
\usepackage{xcolor}
\usepackage{xspace}
\usepackage{enumerate}
\usepackage{geometry}
\usepackage{marginnote}

\usepackage{float}
\usepackage{tikz}
\usepackage{tikz-cd}
\usepackage{wrapfig}
\usetikzlibrary{positioning}

\usepackage{amsmath, amsthm}
\usepackage{amssymb}
\usepackage{todonotes}

\usepackage{cite}

\usepackage{tikz}
\usetikzlibrary{positioning}

\theoremstyle{definition}

\newtheorem{thm}{Theorem}[section]
\newtheorem*{thm*}{Theorem}
\newtheorem{lemma}[thm]{Lemma}

\newtheorem{defn}[thm]{Definition}
\newtheorem{prob}[thm]{Problem}
\newtheorem{claim}[thm]{Claim}
\newtheorem{prop}[thm]{Proposition}
\newtheorem{cor}[thm]{Corollary}

\newtheorem{remark}[thm]{Remark}
\newtheorem{fact}[thm]{Fact}
\newtheorem{question}[thm]{Question}

\newtheorem{Notation}[thm]{Notation}
\newtheorem{obs}[thm]{Observation}
\newtheorem{conj}[thm]{Conjecture}

\renewcommand{\subset}{\subseteq}

\newcommand\force{\Vdash}
\newcommand\R{\mathbb{R}}
\newcommand\N{\mathbb{N}}
\newcommand\Q{\mathbb{Q}}
\renewcommand\P{\mathbb{P}}

\newcommand{\set}[2]{ \left\{ #1 :\, #2 \right\} }
\newcommand{\seqq}[2]{ \left( #1 :\, #2\right) }
\newcommand{\seq}[1]{ \left( #1 \right) }

\title[The first $\omega$ Friedman-Stanley jumps]{
The finite Friedman-Stanley jumps:\\ generic dichotomies for Borel homomorphisms}

\author{Assaf Shani}
\address{Department of Mathematics and Statistics, Concordia University University, Montreal, QC  H3G 1M8, Canada}
\email{assaf.shani@concordia.ca}
\urladdr{https://sites.google.com/view/assaf-shani/}

\date{\today}



\thanks{Research partially supported by NSF grant DMS-2246746.}

\begin{document}

\begin{abstract}
Fix $n=1,2,3,\dots$ or $n=\omega$. We prove a dichotomy for Borel homomorphisms from the $n$-th Friedman-Stanley jump $=^{+n}$ to an equivalence relation $E$ which is classifiable by countable structures: if there is no reduction from $=^{+n}$ to $E$, then in fact all Borel homomorphisms are very far from a reduction.
For this we use a different presentation of $=^{+n}$, equivalent up to Borel bi-reducibility, which is susceptible to Baire-category techniques. 

This dichotomy is seen as a method for proving positive Borel reducibility results from $=^{+n}$. As corollaries we prove: (1) for $n\leq\omega$, $=^{+n}$ is in the spectrum of the meager ideal. This extends a result of Kanovei, Sabok, and Zapletal for $n=1$; (2) $=^{+\omega}$ is a regular equivalence relation. This answers positively a question of Clemens; (3) for $n<\omega$, the equivalence relations, classifiable by countable structures, which do not Borel reduce $=^{+n}$ are closed under countable products. This extends a result of Kanovei, Sabok, and Zapletal for $n=1$.
\end{abstract}

\maketitle

\section{Introduction}\label{section: introduction}
This paper is a contribution to the study of equivalence relations on Polish spaces up to Borel reducibility.
Given equivalence relations $E$ and $F$ on Polish spaces $X$ and $Y$ respectively, a map $f\colon X\to Y$ is said to be a \textbf{reduction} of $E$ to $F$ if for any $x_1,x_2\in X$, 
\begin{equation*}
    x_1\mathrel{E}x_2\iff f(x_1)\mathrel{F}f(x_2).
\end{equation*}
We say that $E$ is \textbf{Borel reducible} to $F$, denoted ${E}\leq_B{F}$ if there is a Borel measurable function which is a reduction of $E$ to $F$. In this case, we think of $E$ as no more complicated than $F$. Borel reducibility is the most central concept in the study of equivalence relations on Polish spaces. Say that $E$ and $F$ are \textbf{Borel bireducible}, denoted $E\sim_B F$, if $E\leq_B F$ and $F\leq_B E$.
An equivalence relation $E$ on a Polish space $X$ is \textbf{Borel} if $E$ is a Borel subset of $X\times X$, with the product topology. More generally, $E$ is \textbf{analytic} if $E$ is an analytic subset of $X\times X$, that is, $E$ is the projection of a Borel subset of $X \times X \times Y$ for some Polish space $Y$.

A central motivation for the field is to study the complexity of various classification problems in mathematics. Generally speaking, separable mathematical objects can be coded as members of some Polish space. Natural notions of equivalence, such as isomorphism between countable graphs, isometry between separable metric spaces, or homeomorphism between compact metric spaces, can then be seen as equivalence relations on Polish spaces. These are generally analytic, and sometiems Borel. Another point of view is the study of (Borel) definable cardinality between quotients of Polish spaces. A reduction of $E$ to $F$ corresponds to an injective map between from quotient space $X/E$ to $Y/F$. So we study injective maps between such quotient spaces, but only consider ``sufficiently nice'' maps, those which lift to a Borel map between Polish spaces.



Given an equivalence relation $E$ on $X$, the \textbf{Friedman-Stanley jump} of $E$ is the equivalence relation $E^+$ on $X^\N$ defined by
\begin{equation*}
    x\mathrel{E^+}y \iff \forall n \exists m (x(n)\mathrel{E} y(n)) \textrm{ and } \forall n \exists m (y(n) \mathrel{E} x(m)),
\end{equation*}
equivalently, if  $\set{[x(n)]_E}{n\in\N}=\set{[y(n)]_E}{n\in\N}$. The quotient $E^+ / X^\N$ may be identified with $\mathcal{P}_{\aleph_0}(E/X)$, the countable powerset of $E/X$. 

The iterated Friedman-Stanley jumps, $=^{+\alpha}$, are defined recursively along the countable ordinals as follows (see \cite[12.2.6]{Gao09}). 
\begin{itemize}
    \item $=^{+0}$ is the equality relation on $\R$, $=_\R$,
    \item $=^{+(\alpha+1)}$ is defined as $(=^{+\alpha})^+$,
    \item $=^{+\lambda}$ is defined as $\prod_{\alpha<\lambda}=^{+\alpha}$, for a limit ordinal $\lambda$.
\end{itemize}
The equivalence relation $=^{+1}$ is often denoted as $=^+$. The Friedman-Stanley jumps play a central role in the theory of equivalence relations. A classification problem is considered ``classifiable using countable sets of reals as complete invariants'' if it is Borel reducible to $=^{+}$; ``classifiable using countable sets of countable sets of reals as complete invariants'' if it is Borel reducible to $=^{+2}$; and so on.

An equivalence relation is \textbf{classifiable by countable structures} if it is Borel reducible to an isomorphism relation on a space of all countable $\mathcal{L}$-structures, for some countable language $\mathcal{L}$. (See \cite[12.3]{Kano08}, \cite[3.6]{Gao09}, \cite{Hjo00}). A Borel equivalence relation which is classifiable by countable structures is Borel reducible to $=^{+\alpha}$ for some countable ordinal $\alpha$ (see~\cite[Theorem~1.5]{Friedman2000}).

When studying some equivalence relation $E$, we would like to compare it, in terms of Borel reducibility, to a given Friedman-Stanley jump $=^{+\alpha}$. The results in \cite{HKL98} provide a powerful tool for proving that $E$ \emph{is} Borel reducible to $=^{+\alpha}$. There are flexible tools to prove \emph{irreducibility} results between some $E$ and $=^{+\alpha}$, such as the study of pinned cardinals~\cite{Larson-Zapletal-Geometric-2020,Ulrich-Rast-Laskowski-2017} and the use of symmetric models in \cite{Sha20}.

\begin{prob}\label{Problem : Main}
Fix a countable ordinal $\alpha$. Develop tools to construct a Borel reduction from $=^{+\alpha}$ to some other equivalence relation.
\end{prob}
In this paper we provide such tools for $\alpha\leq\omega$. First,  we note that Problem~\ref{Problem : Main} is well understood for $=^+$, that is, $\alpha=1$. There are many results reducing $=^+$ to other equivalence relations\footnote{In the literature, $=^+$ takes many names, including $\mathrm{Eq}^+$, $E_{\mathrm{ctbl}}$, $F_2$, and $\cong_2$.}, for example, \cite[Theorem~65~part 2]{For96}, \cite[Theorem~1.1]{Kaya2017-minimal-cantor-systems}, and \cite[Proposition~3.5]{Anti-classification-CMMS}. There are also three general results for constructing such a reduction:
\begin{itemize}
    \item Marker~\cite[Theorem~1.2]{Mar07} provides a model theoretic criterion for a first order isomorphism relation $\cong_T$ to reduce $=^+$: if the type space $S(T)$ is uncountable.
    \item Larson and Zapletal~\cite[Theorem~2.8.11]{Larson-Zapletal-Geometric-2020} provide a set theoretic criterion for an analytic equivalence relation $E$ to reduce $=^+$: if $E$ is unpinned in the Solovay extension.
    \item Kanovei, Sabok, and Zapletal provided the following Baire-category tool. 
\end{itemize} 
\begin{thm}[Kanovei-Sabok-Zapletal~{\cite[Theorem 6.24]{ksz}}]\label{thm;ksz-thm-6-24}
Let $E$ be an analytic equivalence relation. Then either
\begin{itemize}
    \item $=^+$ is Borel reducible to $E$, or
    \item any Borel homomorphism from $=^+$ to $E$ maps a comeager set into a single $E$-class. 
\end{itemize}
\end{thm}
Given equivalence relations $E$ and $F$ on Polish spaces $X$ and $Y$, a map $f\colon X\to Y$ is a \textbf{Borel homomorphism} from $E$ to $F$, denoted $f\colon E \to_B F$, if for any $x_1,x_2\in X$,
\begin{equation*}
    x_1\mathrel{E}x_2 \implies f(x_1)\mathrel{F}f(x_2).
\end{equation*}
Theorem~\ref{thm;ksz-thm-6-24} says that if there is no Borel reduction of $=^+$ to $E$, then in fact all Borel homomorphisms from $=^+$ to $E$ are trivial, on a comeager set.

We mention two immediate difficulties in generalizing this to the higher jumps.
\begin{remark}
For $n\geq 2$, $=^{+n}$ \emph{does not} behave well, in terms of Baire-category, with the product topology given by the Friedman-Stanley jump. See Claim~\ref{claim: =++ reduction on comeager} below.
\end{remark}
\begin{remark}
    For $1\leq k<n$, there \emph{is} a natural Borel homomorphism from $=^{+n}$ to $=^{+k}$, which is not ``completely trivial''. For example, the homomorphism $u\colon (\R^\N)^\N\to  \R^\N$ from $=^{+2}$ to $=^+$, defined by $u(x)(\seq{n,m}) = x(n)(m)$, where $\seq{\,,\,}\colon \N\times\N\to\N$ is a bijection. This is the ``union homomorphism'', taking a set of sets of reals $\set{\set{x(n)(m)}{m\in\N}}{n\in\N}$ to an enumeration of their union $\set{x(n)(m)}{n,m\in\N}$.
\end{remark}
With these two modifications in mind, we provide a complete Baire-category analysis of all Borel homomorphisms from $=^{+n}$, $n\leq\omega$, to equivalence relations which are classifiable by countable structures. This is the main result of the paper.
\begin{thm}\label{Theorem : Main}
There are Borel equivalence relations $F_n$ for $n\leq \omega$, and Borel homomorphisms $u^n_k\colon F_n\to_B F_k$ for $1\leq k\leq n\leq\omega$, so that for each $n\leq \omega$, $F_n$ and $=^{+n}$ are Borel bireducible, 
and the following dichotomy for Borel homomorphisms holds.
For any equivalence relation $E$ which is classifiable by countable structures, either
  \begin{itemize}
        \item $F_{n}$ is Borel reducible to $E$, or
        \item  for any Borel homomorphism $f\colon F_n\to_B E$ there is some $k<n$ so that the homomorphism $f$ factors through $u^n_{k}$ on a comeager set, that is, there is a Borel homomorphism $h\colon F_{k}\to_B E$, defined on a comeager set, so that $h\circ u^n_{k}(x)\mathrel{E}f(x)$ for a comeager set of $x$ in the domain of $F_n$. (See Figure~\ref{figure: main thm factoring}.)
    \end{itemize}
\end{thm}

\begin{figure}[H]\label{figure: main thm factoring}
    \centering
\begin{tikzpicture}[
node/.style={}]
\node[node]      (dplus)   at (0,0)        {$F_n$};
\node[node]      (eqdplux) [right=0.01cm of dplus]  {\color{gray}{$\sim_B\, =^{+n}$}};

{
\node[node]     (uplus) [below=of dplus]
{$F_k$};
\draw[->]  (dplus.south) --node[anchor=east] {$u^n_k$} (uplus.north) ;
}
{
\node[node]     (fplus) [right=of uplus]
{$E$};

\draw[->] (dplus.south east) -- node[anchor=south west] {$f$} (fplus.north west)   ;
\draw[->,dashed]    (uplus.east) --node[anchor=south] {$h$} (fplus.west) ;
}
\end{tikzpicture}


    \caption{$(\forall f\colon F_n\to_B E )(\exists k<n\, \exists h\colon F_k\to_B E)$} 
    \label{fig:enter-label}
\end{figure}
\begin{remark}
As remarked above the equivalence relations $F_n$, $n>1$, are necessarily different than $=^{+n}$. For $n=1$, the equivalence relation $F_1$ is simply $=^+$. Let $F_0$ be the trivial equivalence relation on a space $\{\ast\}$ with a single element, and consider the trivial homomorphism $u^1_0\colon F_1 \to_B F_0$. Then the second bullet of Theorem~\ref{thm;ksz-thm-6-24} for $=^+$ has the same form as Figure~\ref{figure: main thm factoring} with $k=0$ and $n=1$. 
\end{remark}

\begin{remark}
    We see Theorem~\ref{Theorem : Main} as a tool to prove that $=^{+n}$ is Borel reducible to some equivalence relation $E$. In order to prove that such reduction exists, it suffices to find a Borel homomorphism which is ``sufficiently different'' from the homomorphisms $u^n_k$, for $k<n$.
\end{remark}

The definition of the equivalence relations $F_n$, appearing in Theorem~\ref{Theorem : Main}, is given in Section~\ref{Section: def of Fn into}. A group action inducing $F_n$ is presented in Section~\ref{subsection : orbit ER presentation}. We then prove the following corollaries of Theorem~\ref{Theorem : Main}. The definitions and background are presented in each subsection.
\begin{enumerate}
    \item {(Section~\ref{subsection : spectrum of the meager ideal}.)}  For $n\leq\omega$, $=^{+n}$ is in the spectrum of the meager ideal. This was proved for $n=1$ in \cite{ksz}.
    \item {(Section~\ref{subsec: primeness})} $=^{+\omega}$ is regular. This answers positively a question of Clemens~\cite{Clemens-primeness-2020}. 
    \item {(Section~\ref{subsec: products})} Fix $n<\omega$. Suppose $G_k$, $k\in\N$, are classifiable by countable structures and ${=^{+n}}\not\leq_B {G_k}$. Then ${=^{+n}}\not\leq_B\prod_{k\in\N}G_k$. This was proved for $n=1$ in \cite{ksz}.
\end{enumerate}
Several open questions related to these results are posed in the relevant subsections. In Section~\ref{subsec: Borel complexity} we note that the Borel complexity of $F_n$ is $\mathbf{\Pi}^0_{2+n}$, which is the optimal potential complexity of $=^{+n}$ by~\cite{HKL98}.

We first focus on proving a corollary of Theorem~\ref{Theorem : Main}, that $F_n$ preserves its complexity on comeager sets (see Section~\ref{subsection : spectrum of the meager ideal}), which is proved in Section~\ref{section: complexity on comeager sets} (Theorem~\ref{thm: warmup proof retains complexity on comeager sets}). In Section~\ref{Section: obstacles first round} we sketch some ideas from \cite{ksz}, for proving that $=^+$ retains its complexity on comeager sets, and explain the main difficulties towards $n\geq 2$. The main construction, which will eventually lead to the necessary reductions of $F_n$, is presented in Section~\ref{section: the main construction}. In Section~\ref{Section: permutations} we present some technical results regarding Vaught transoforms for the actions presented in Section~\ref{subsection : orbit ER presentation}.

In Section~\ref{Section: obstacles second round} we present some ideas behind the proof of Theorem~\ref{thm;ksz-thm-6-24} from~\cite{ksz}, and explain the remaining difficulties towards extending these to the $n\geq 2$ case. In particular, we will use the following lemma.
\begin{lemma}\label{lemma: breaking down homomorphism}
Let $E$ be an equivalence relation which is classifiable by countable structures and let $f\colon F_n \to_B E$ be a Borel homomorphism which does not factor through $u^n_k$, $k<n$, on a comeager set. Then there are equivalence relations $E_k$, for $k<n$, Borel homomorphisms $\pi^n_k \colon E \to_B E_k$, $\pi^{k+1}_k\colon E_{k+1} \to_B E_k$, and $f_k\colon F_k \to_B E_k$ so that the following diagram commutes on comeager sets, and so that $f_k$ does not factor through $u^k_l$ for $l<k$.
\begin{center}
    \begin{tikzpicture}
    \node[draw=none] (F1) at (0,0) {$F_1$};
    \node[draw=none] (F2) at (2,0) {$F_2$};
    \node[draw=none] (F3) at (4,0) {$F_3$};
    \node[draw=none] (F4) at (6,0) {$\dots$};
    \node[draw=none] (F5) at (8,0) {$F_n$};
    \node[draw=none] (E1) at (0,-2) {$E_1$};
    \node[draw=none] (E2) at (2,-2) {$E_2$};
    \node[draw=none] (E3) at (4,-2) {$E_3$};
    \node[draw=none] (E4) at (6,-2) {$\dots$};
    \node[draw=none] (E5) at (8,-2) {$E$};
    \draw[->] (F2) -- node[midway,above] {$u^2_1$} (F1);
    \draw[->] (F1) -- node[midway,right] {$f_1$} (E1);
    \draw[->] (F2) -- node[midway,right] {$f_2$} (E2);
    \draw[->] (F3) -- node[midway,above] {$u^3_2$} (F2);
    \draw[->] (F3) -- node[midway,right] {$f_3$} (E3);
    \draw[->] (F4) -- node[midway,above] {$u^4_3$} (F3);
    \draw[->] (F5) -- node[midway,above] {$u^n_k$} (F4);
    \draw[->] (F5) -- node[midway,right] {$f$} (E5);
    \draw[->] (E2) -- node[midway,above] {$\pi^2_1$} (E1);
    \draw[->] (E3) -- node[midway,above] {$\pi^3_2$} (E2);
    \draw[->] (E4) -- node[midway,above] {$\pi^3_2$} (E3);
    \draw[->] (E5) -- node[midway,above] {$\pi^n_k$} (E4);
\end{tikzpicture}
\end{center}
\end{lemma}
The proof of Theorem~\ref{Theorem : Main} is then completed in Section~\ref{Section: proof of main thm}.
\begin{remark}
The proof of Lemma~\ref{lemma: breaking down homomorphism} is the only place in which we use that $E$ is classifiable by countable structures. Extending the lemma for a wider class of equivalence relations will similarly extend Theorem~\ref{Theorem : Main}, as well as Theorem~\ref{Theorem : =+omega prime} and Proposition~\ref{prop: nonreduction for products} below.
\end{remark}
\begin{question}\label{question: breaking homomorphism}
    Is Lemma~\ref{lemma: breaking down homomorphism} true for all analytic equivalence relations?
\end{question}

\begin{prob}
    Find a model theoretic condition for an isomorphism relation $\cong_T$ to reduce $=^{+n}$, extending the result \cite[Theorem~1.2]{Mar07} for $n=1$.
\end{prob}
\begin{prob}
    Find a set theoretic condition for an equivalence relation $E$ to reduce $=^{+n}$, extending the result \cite[Theorem~2.8.11]{Larson-Zapletal-Geometric-2020} for $n=1$.
\end{prob}

\subsection{The definition of $F_n$ and $u^n_k$ from Theorem~\ref{Theorem : Main}}\label{Section: def of Fn into}
Consider the Polish space $((2^\N)^\N)^\omega$, with the natural product topology. We use the following standard notation: the space $2$ is identified with the discrete space with two elements $\{0,1\}$. The ordinal $\omega$ is identified with the set of natural numbers $\N=\{0,1,2,\dots\}$.
The space $2^\N$ is identified with the space $\mathcal{P}(\N)$ of all subsets of $\N$.

Given $x\in ((2^\N)^\N)^\omega$, we define a sequence $A^x_n$, $n=1,2,\dots$, as follows.
\begin{itemize}
    \item $A^x_1=\set{x(0)(k)}{k\in\N}$.
    \item For $l\in\N$, define $a^{x,l}_1=\set{x(0)(k)}{x(1)(l)(k)=1}$, a subset of $A^x_1$.
\end{itemize}

Given $A^x_n$ and $a^{x,l}_n$ for $l\in\N$, define

\begin{itemize}    
    \item $A^x_{n+1}=\set{a^{x,k}_n}{k\in\N}$, and
    \item $a^{x,l}_{n+1}=\set{a^{x,k}_n}{x(n+1)(l)(k)=1}$, a subset of $A^x_{n+1}$.
\end{itemize}
For $m<\omega$, and $x\in ((2^\N)^\N)^m$ we define similarly $A^x_n$ for $n=1,2,\dots,m$.
\begin{defn}\label{defn : domain of Fn}
For $2\leq n\leq \omega$, define $X_n\subset ((2^\N)^\N)^n$ as the set of all $x$ such that:
\begin{enumerate}
    \item $(\forall 1\leq i<n)(\forall m)(\exists k)x(i)(k)(m)=1$;
    \item $(\forall 1\leq i <n)(\forall k)(\exists m)x(i)(k)(m)=1$;
    \item $(\forall 1\leq i<n)(\forall k,l_1,l_2)(x(i-1)(l_1)\mathrel{=}{}x(i-1)(l_2)\rightarrow x(i)(k)(l_1)=x(i)(k)(l_2))$.
\end{enumerate}
\end{defn}
\begin{obs}
From condition (1) it follows that for any $n\leq\omega$,
\begin{itemize}
    \item If $x\in X_n$ then $A^x_k=\bigcup A^x_{k+1}$ for any $k<n$.
    \item For $m<\omega$, $m\leq n$, $x,y\in X_n$, if $A^x_m=A^y_m$ then $A^x_k=A^y_k$ for all $k<m$.
\end{itemize}
\end{obs}
\begin{obs}
In the construction above, we used the binary sequence $x(i+1)(t)\in 2^\N$ to code a subset of $A^x_{i+1}$, via its enumeration $\seqq{a^{x,k}_i}{k\in\N}$. Condition (3) says that for $x\in X_n$, if $l_1$ and $l_2$ are identified in this enumeration, $a_{i}^{x,l_1}=a_{i}^{x,l_2}$, then they are also identified by $x(i+1)(t)$.
Condition (2) says that for $x\in X_n$, for each $i<n$, $a_{i+1}^{x,l}$ is a non-empty subset of $A^x_{i+1}$.
\end{obs}
\begin{remark}
Note that $X_n$ is a dense $G_\delta$ subset of $((2^\N)^\N)^n$. For condition (3) this is true since for a dense $G_\delta$ set of $x\in ((2^\N)^\N)^n$, $x(i-1)(l_1)\neq x(i-1)(l_2)$, for $l_1\neq l_2$.
\end{remark}
\begin{defn}[Main definition: $F_n$ and $u^n_k$]\label{defn: main definition} \,\\
(1) For $n<\omega$, the equivalence relation $F_n$ is defined on $X_n$ by
\begin{equation*}
    x\mathrel{F_n}y\iff A^x_n=A^x_n.
\end{equation*}
The equivalence relation $F_\omega$ is defined on $X_\omega$ by
\begin{equation*}
    x\mathrel{F_\omega}y\iff \forall n<\omega(A^x_n=A^y_n).
\end{equation*}
(2) Given $m\leq n$, define $u^n_m\colon X_n\to X_m$ as the natural projection map to the first $m$ copies of $(2^\N)^\N$. Then $u^n_m\colon F_n\to_B F_m$ is a Borel homomorphism. For $n<\omega$, the homomorphism $u^n_{n-1}$ can be seen as the union map, sending the set $A^x_n$ to its union $A^x_{n-1}=A^{u^n_{n-1}(x)}_{n-1}$.
\end{defn}
The map $x\mapsto A^x_n$ can be seen as a reduction of $F_n$ to $=^{+n}$. Specifically, there is a Borel map from $X_n$ to $(2^\N)^{\omega^n}$ (the domain of $=^{+n}$), sending each $x\in X_n$ to some $z\in(2^\N)^{\omega^n}$ ``enumerating'' the set $A^x_n$.


\subsection{Group action}\label{subsection : orbit ER presentation}
In this section we present the equivalence relations $F_n$ as orbit equivalence relation, when restricted to a (large) subdomain.\footnote{In this section, and throughout the paper, we use colors for emphasis and clarification. The reader is advised to view these pages in color.}

Consider the Polish group $S_\infty$ of all permutations of $\N$, with its natural action $a\colon S_\infty \,{\color{blue}\curvearrowright}\, (2^\N)^{\color{blue}\N}$, permuting the sequence of reals. The induced orbit equivalence relation on $(2^\N)^\N$ is not $=^+$, but is Borel bireducible with $=^+$ (see \cite[Exercise 8.3.4]{Gao09}). The two equivalence relations in fact agree on the comeager set of all injective sequences of reals.

Consider also the natural action $b_0\colon S_\infty\curvearrowright 2^\N$, permuting binary sequence. Let $b\colon S_\infty \,{\color{blue}\curvearrowright}\, (2^{\color{blue}\N})^\N$ be the diagonal action, $g\cdot_b (x_n)_n=(g\cdot_{b_0}x_n)_n$.
Recall the definition of $F_2$ on $(2^\N)^\N\times (2^\N)^\N$. When permuting a sequence of reals (the first coordinate) using $a$, the action of $b$ on the second coordinate updates the binary sequences, so that they still carve out the same subset of reals as before. To recover all the symmetries of $F_2$, we also want to allow $S_\infty$ to act via $a$ on the second coordinate.

Note that the actions $a$ and $b$ on $(2^\N)^\N$ commute, and so give rise to the product action $c=(b,a)$ of the product group $S_\infty\times S_\infty$
\begin{equation*}
        c\colon {\color{red}S_\infty}\times {\color{blue}S_\infty} \curvearrowright (2^{\color{red}\N})^{\color{blue}\N}
\end{equation*}
Define an action
\begin{equation*}
    a_2\colon {\color{red}S_\infty}\times {\color{blue}S_\infty} \curvearrowright (2^\N)^{\color{red}\N}\times (2^{\color{red}\N})^{\color{blue}\N}
\end{equation*}
by 
\begin{equation*}
 (g,h)\cdot_{a_2}x=(g\cdot_a x(0), (g,h)\cdot_c x(1)).   
\end{equation*}

The corresponding orbit equivalence relation agrees with $F_2$ on the large, comeager set, of all $x\in (2^\N)^\N\times (2^\N)^\N$ for which $x(0)$ is an injective enumeration of $A^x_1$ and $\seqq{a^{x,l}_1}{l\in\N}$ is an injective enumeration of $A^x_2$ (recall the definitions from Section~\ref{section: introduction}). More generally:

\begin{defn} For $n\leq\omega$ define
\begin{equation*}
    a_n\colon (S_\infty)^n\curvearrowright ((2^\N)^\N)^n
\end{equation*}
so that for $g\in(S_\infty)^n$ and $x\in ((2^\N)^\N)^n$
\begin{equation*}
    (g\cdot_{a_n}x)(k+1)=(g(k),g(k+1))\cdot_c x(k+1)
\end{equation*}
and $(g\cdot_{a_n} x)(0)=g(0)\cdot_a x(0)$.
\end{defn}
Let $X_n^{\mathrm{inj}}$ be the set of all $x\in X_n$ so that $x(0)$ is an injective enumeration of $A^1_x$ and for each $k<n$, $\seqq{a^{x,l}_k}{l\in\N}$ is an injective enumeration of $A^x_{k+1}$.
\begin{claim}
For each $n\leq \omega$.
\begin{enumerate}
    \item $X_n^{\mathrm{inj}}$ is a comeager subset of $X_n$, and is $a_n$-invariant.
    \item On $X_n^{\mathrm{inj}}$ the orbit equivalence relation induced by $a_n$ is $F_n$.
\end{enumerate}
\end{claim}
\begin{remark}
The group action provides another point of view that the presentation $F_n$ is better behaved than $=^{+n}$. For example, $=^{+2}$, defined on $((2^\N)^\N)^\N$ as the Friedman-Stanley jump of $=^+$, is naturally induced (on a subdomain) by an action of the infinite support wreath product group $S_\infty\wr S_\infty$. This group is defined as the semi-direct product $S_\infty \ltimes (S_\infty)^\N$ with the natural permutation action of $S_\infty\,{\color{blue}\curvearrowright}\, (S_{\infty})^{\color{blue}\N}$. Similarly, the higher jumps $=^{+n}$ can be presented (on a subdomain) as an orbit equivalence relation induced by a natural action of an iterated wreath product of $S_\infty$. (See~\cite[Proposition~2.3]{CC19} for example, where variations of the Friedman-Stanley jump are considered.)
\end{remark}

\subsection{The spectrum of the meager ideal}\label{subsection : spectrum of the meager ideal}

\begin{defn}[Kanovei-Sabok-Zapletal~{\cite[Definition~1.16]{ksz}}\footnote{Kanovei, Sabok, and Zapletal studied the behavior of equivalence relations on $I$-positive sets for various ideals $I$. Here we only mention the case where $I$ is the meager ideal.}]
An analytic equivalence relation $E$ is {\bf in the spectrum of the meager ideal} if there is an equivalence relation $F$ on a Polish space $Y$ so that
\begin{itemize}
    \item $E$ and $F$ are Borel bireducible;
    \item For any non-meager set $C\subset Y$, $F\restriction C$ is Borel bireducible with $F$.
\end{itemize}
For $F$ as in the second bullet, we say that {\bf $F$ retains its complexity on non-meager sets}.
\end{defn}
Kanovei, Sabok, and Zapletal~\cite{ksz} concluded from Theorem~\ref{thm;ksz-thm-6-24} that $=^+$, on $\R^\N$, retains its complexity on non-meager sets, and is therefore in the spectrum of the meager ideal. The higher jumps \emph{do not} retain their complexity, with the topology coming from the jump operation.

\begin{claim}\label{claim: =++ reduction on comeager}
    There is a comeager set $C$ so that $({=^{+2}}\restriction C)\leq_B{=^+}$.
\end{claim}
\begin{proof}
Recall that $=^{+2}$ is defined on the space $(\R^\N)^\N$. Let $C\subset (\R^\N)^\N$ be the set of all $x\in (\R^\N)^\N$ so that for any $n,m,l,k\in\N$, $(n,m)\neq (l,k)\implies x(n)(m)\neq x(l)(k)$. $C$ is a comeager subset of $(\R^\N)^\N$. 

Define $g\colon C\to (\R\times\R)^{\N^3}$ by $g(x) = \seqq{(x(n)(m),x(n)(k))}{n,m,k\in\N}$. Fix a bijection $e\colon \N\to \N^3$, which extends naturally to a homeomorphism $\hat{e}\colon (\R\times\R)^{\color{blue}\N^3}\to(\R\times\R)^{\color{blue}\N}$. Define $f\colon C\to(\R\times\R)^\N$ by $f = \hat{e}\circ h$.

For $x\in D$, the sets $\set{x(n)(m)}{m\in\N}$ are disjoint for different values of $n$. $f(x)$ is an enumeration of the equivalence relation partitioning the set of reals $\set{x(n)(m)}{n,m\in\N}$ into the sets $\set{\set{x(n)(m)}{m\in\N}}{n\in\N}$. It follows that $f$ is a reduction of $({=^{+2}}\restriction C)$ to ${(=_{\R\times\R})^+}$. Since ${=_\R} \sim_B {=_{\R\times\R}}$, we conclude that $({=^{+2}}\restriction C)\leq_B{=^+}$, as required. 
\end{proof}

\begin{thm}\label{Theorem : omega FS jumps in specturm}
For each $1\leq n\leq\omega$, $=^{+n}$ is in the spectrum of the meager ideal.
\end{thm}
Since ${F_n} \sim_B {=^{+n}}$, it suffices to prove the following.
\begin{prop}
$F_n$ on $X_n$ retains its complexity on non-meager sets.
\end{prop}
\begin{proof}
First we make the following two observations.
\begin{enumerate}
    \item For each $n\leq\omega$, each $F_n$ class is meager in $X_n$.
    \item For any $k<n\leq\omega$, $u^n_k$ is not a reduction on any non-meager set.
\end{enumerate}
Fix a non-meager set $Z\subset X_n$. We prove that ${F_n}\leq_B{F_n\restriction Z}$. First, we claim that there is a Borel homomorphism $f\colon F_n\to_B F_n\restriction Z$ which is a reduction on a non-meager set. For such $f$, the second bullet of Theorem~\ref{Theorem : Main} fails: if $f$ factors through $u^n_k$ on a comeager set, for $k<n$, then it would follow that $u^n_k$ is a reduction on a non-meager set, contradicting (2) above. 
We then conclude, by Theorem~\ref{Theorem : Main}, that $F_n$ is Borel reducible to $F_n\restriction Z$.

Note that if $Z$ is $F_n$-invariant, it is easy to find a homomorphism $f$ as claimed: simply let $f$ be the identity on $Z$, and a constant function outside of $Z$. In general, we can find such a homomorphism using large section uniformization, as follows. We will use below category quantifiers and the Vaught transform. See \cite[8.J]{Kechris-DST-1995} and \cite[3.2]{Gao09}.

Recall that, once restricted to a comeager invariant set $X_n^{\mathrm{inj}}$, $F_n$ can be presented as an orbit equivalence relation induced by a continuous action of the Polish group $G=(S_\infty)^n$ (see Section~\ref{subsection : orbit ER presentation}). We may assume that $Z\subset X_n^{\mathrm{inj}}$. 
Fix a countable dense set $G_0\subset G$. Since almost every orbit is dense, the set $Z'= G_0\cdot Z$ is comeager. 
For any $g\in G$, $g^{-1}Z'$ is comeager, that is, $\forall^\ast x\in X (g\cdot x\in Z')$. We conclude that $\forall^\ast x\in X\,\forall^\ast g\in G(g\cdot x\in Z')$. That is, the invariant set $B=\set{x}{\forall^\ast g\in G(g\cdot x\in Z')}$ is comeager. By \cite[Theorem~18.6]{Kechris-DST-1995} there is a Borel map $h\colon B\to G$ so that $h(x)\cdot x\in Z'$ for all $x\in B$.

Finally, define $f\colon X\to Z$ as follows. Fix $z_0\in Z$ and an enumeration $(\gamma_n)_{n\in\mathbb{N}}$ of $G_0$. If $x\in X\setminus B$, $f(x)=z_0$. If $x\in B$, define $f(x)=\gamma_n\cdot h(x)\cdot x$ for the minimal $n$ so that $\gamma_n\cdot h(x)\cdot x\in Z$. Then $f$ is a Borel homomorphism as claimed.
\end{proof}

\begin{conj}
For each countable ordinal $\alpha$, $=^{+\alpha}$ is in the spectrum of the meager ideal.
\end{conj}
\begin{question}
Is the spectrum of the meager ideal closed under
\begin{enumerate}
    \item the Friedman-Stanley jump operation;
    \item countable products.
\end{enumerate}
\end{question}
\subsection{A question of Clemens}\label{subsec: primeness}
In the context of definable cardinality of quotients of Polish spaces, a Borel homomorphism corresponds to a definable map between two such quotients, and a Borel reduction corresponds to an injective definable map.

\begin{defn}[Clemens~\cite{Clemens-primeness-2020}] Let $E$ and $F$ be Borel equivalence relations on Polish spaces $X$ and $Y$ respectively. Say that {\bf $E$ is prime to $F$} if for any Borel homomorphism $f\colon E\to_B F$, $E$ retains its complexity on a fiber, that is, there is $y\in Y$ so that $E$ is Borel reducible to $E\restriction \set{x\in X}{f(x)\mathrel{F}y}$.
\end{defn}
Primeness is a strong form of Borel-irreducibility, which holds between many pairs of benchmark equivalence relations (see~\cite[Theorem~1]{Clemens-primeness-2020}).

In the classical context of cardinality, primeness corresponds to a pigeonhole principle: any function $f\colon A\to B$ has a fiber of cardinality $|A|$. This is true if and only if the cardinality $|B|$ is strictly smaller than the cofinality of $|A|$. Recall that the cardinality $|A|$ is regular if it is equal to its cofinality, that is, if for any $|B|<|A|$, any function from $A$ to $B$ has a fiber of size $|A|$.

Following this analogy Clemens defined \textbf{regular} equivalence relation as follows.
\begin{defn}[Clemens~\cite{Clemens-primeness-2020}]\label{Definition : prime} 
A Borel equivalence relation $E$ is {\bf regular} if for any Borel equivalence relation $F$, if ${F}<_B{E}$ then $E$ is prime to $F$.
\end{defn}

Clemens~\cite[Question~7.3]{Clemens-primeness-2020} asked if $=^{+\omega}$ is regular. We confirm this. (Notational warning: what we call here $=^{+\omega}$ is denoted by $\mathbb{F}_\omega$ in \cite{Clemens-primeness-2020}.)
\begin{thm}\label{Theorem : =+omega prime}
For any equivalence relation $E$ which is classifiable by countable structures, either ${=^{+\omega}}\leq_B {E}$ or $=^{+\omega}$ is prime to $E$. In particular, $=^{+\omega}$ is regular.
\end{thm}
\begin{proof}
Note that all the properties above respect Borel bireducibility. In particular, if $E\sim_B E'$ then $E$ is prime to $F$ if and only if $E'$ is prime to $F$. Therefore, it suffices to prove the theorem with $F_\omega$ instead of $=^{+\omega}$.

Fix $E$ as in the theorem and assume that $F_\omega$ is not Borel reducible to $E$. Let $f\colon F_\omega\to_B E$ be a Borel homomorphism. By Theorem~\ref{Theorem : Main} there is some $k<\omega$, a Borel homomorphism $g\colon F_k\to_B E$, defined on a comeager set, and a comeager $C\subset X_\omega$, so that for any $x\in C$,
\begin{equation*}
    g(u^\omega_k(x))\mathrel{E}f(x).
\end{equation*}
View $((2^\N)^\N)^\omega$ as $((2^\N)^\N)^k \times ((2^\N)^\N)^{\omega\setminus k}$. Recall that $u^\omega_k$ is the projection from $((2^\N)^\N)^k \times ((2^\N)^\N)^{\omega\setminus k}$ to $((2^\N)^\N)^k$. By the Kuratowski-Ulam theorem (see \cite[Theorem~8.41~(iii)]{Kechris-DST-1995}) there is $y\in ((2^\N)^\N)^k$ so that $C_y=\set{z\in ((2^\N)^\N)^{\omega\setminus k}}{(y,z)\in C}$ is comeager in $((2^\N)^\N)^{\omega\setminus k}$.
Note that $\{y\}\times C_y$ is contained in the fiber $\set{x\in X_\omega}{f(x)\mathrel{E}g(y)}$. We will finish the proof by showing that $F_\omega$ is Borel reducible to $F_\omega\restriction \{y\}\times C_y$.

Consider the homeomorphism $\phi\colon ((2^\N)^\N)^{\omega\setminus k}\to ((2^\N)^\N)^{\omega}$, defined by $\phi(z)(l)=z(l+k)$. Then for $z_1,z_2\in ((2^\N)^\N)^{\omega\setminus k}$, 
\begin{equation*}
    (y,z_1)\mathrel{F_\omega}(y,z_2)\iff \phi(z_1)\mathrel{F_\omega}\phi(z_2).
\end{equation*}
The set $\phi(C_y)$ is comeager in $((2^\N)^\N)^{\omega}$, as $\phi$ is a homeomorphism. Since $F_\omega$ retains its complexity of comeager sets, there is a Borel reduction $h\colon F_\omega\to F_\omega\restriction\phi(C_y)$. Finally, the map
\begin{equation*}
    x\mapsto (y,\phi^{-1}(h(x)))
\end{equation*}
is a Borel reduction of $F_\omega$ to $F_\omega\restriction \{y\}\times C_y$, as required.
\end{proof}
Clemens~\cite[Lemma~7.6]{Clemens-primeness-2020} showed that if $\alpha\geq 2$ is not of the form $\omega^\beta$, for some countable ordinal $\beta$, then $=^{+\alpha}$ is not regular.
\begin{question}[See {\cite[Question~7.3]{Clemens-primeness-2020}}]
For a countable ordinal $\beta$, is $=^{+\omega^\beta}$ regular?
\end{question}

As in Definition~\ref{Definition : prime}, Clemens defined an equivalence relation $E$ as \textbf{prime} if for any Borel equivalence relation $F$, either ${E}\leq_B{F}$ or $E$ is prime $F$. In the context of definable cardinality, when not every two sizes are comparable, this is a strengthening of being regular. A positive answer to Question~\ref{question: breaking homomorphism} will imply that $=^{+\omega}$ is prime.
\begin{question}
    Is $=^{+\omega}$ prime?
\end{question}

\subsection{Non-reduction to products}\label{subsec: products}
Given equivalence relations $E_k$ on $X_k$, the \textbf{product equivalence relation} $\prod_k E_k$ is defined on the space $\prod_k X_k$ by
\begin{equation*}
    x\,\mathrel{\prod_k E_k}\,y\iff x(k)\,\mathrel{E_k}\,{y(k)}\textrm{ for all }k.
\end{equation*}
We write $E^\N$ for the product $\prod_k E_k$ where $E_k=E$ for all $k$.
The product operation plays an important role in the study of Borel equivalence relations. For example, it follows from the dichotomy theorem proved by Hjorth and Kechris~\cite{Hjorth-Kechris-recent-developments-2001} that the equivalence relation $E_0^\N$, also known as $E_3$, is an immediate successor of $E_0$ with respect to $\leq_B$.
When studying jump operations on Borel equivalence relations, a product is often used to define the limit stages of iterated jumps. The definition of $=^{+\omega}$ as $\prod_n =^{+n}$ is one such example.
The following result shows that, for $n<\omega$, $=^{+n}$ cannot be presented as a product of strictly simpler equivalence relations.
\begin{prop}\label{prop: nonreduction for products}
Fix $n<\omega$. For $k<\omega$, let $E_k$ be an equivalence relation, classifiable by countable structures, so that ${=^{+n}}\not\leq_B{E_k}$. Then ${=^{+n}}\not\leq_B{\prod_k E_k}$.
\end{prop}
\begin{proof}
We may replace $=^{+n}$ by $F_n$. Note that a Borel homomorphism $f\colon F_n\to_B\prod_k E_k$ can be identified with a sequence of Borel homomorphisms $f_k\colon F_n\to_B E_k$. By Theorem~\ref{Theorem : Main}, each $f_k$ factors, on a comeager set, through $u^n_{n-1}$. It follows that $f$ factors through $u^n_{n-1}$ on a comeager set. In particular, a Borel homomorphism $f\colon F_n\to \prod_k E_k$ cannot be a reduction.
\end{proof}
This was proved (for all analytic equivalence relations) for $n=1$ by Kanovei, Sabok, and Zapletal~\cite[Corollary~6.30]{ksz}. The result is phrased there in terms of intersections of equivalence relations. Given equivalence relations $E_k$ on a common space $X$, let their intersection $\bigcap_k E_k$ be the equivalence relation on $X$ defined by    $x\,\mathrel{\bigcap_k E_k}\,y \iff x\mathrel{E_k}y\textrm{ for every }k$. 

There is a close relationship between products and intersections.
Note that $\prod_k E_k$ can be written as an intersection of equivalence relations $E'_k$ on $\prod_k X_k$ so that $E'_k\sim_B E_k$ for each $k$. Furthermore, the intersection $\bigcap_k E_k$ is Borel reducible to $\prod_k E_k$, witnessed by the diagonal map $X\to X^\N$, $x\mapsto\seq{x,x,\dots}$. Therefore Proposition~\ref{prop: nonreduction for products} is equivalent to a similar result for intersections:
\begin{cor}
Fix $n<\omega$ and a Polish space $X$. For $k<\omega$, let $E_k$ be an equivalence relation on $X$, classifiable by countable structures, so that ${=^{+n}}\not\leq_B{E_k}$. Then ${=^{+n}}\not\leq_B{\bigcap_k E_k}$.
\end{cor}


\subsection{Borel complexity}\label{subsec: Borel complexity}
The equivalence relations $=^{+n}$, $n<\omega$, are naturally written as $\mathbf{\Pi}^0_{2n+1}$ relations on their domains, where $=^+$ is $\mathbf{\Pi}^0_3$, and each application of the Freidman-Stanley jump operator adds an alternating $\forall \, \exists$ quantification. The equivalence relations $=^{+n}$ are in fact simpler, in terms of \textbf{potential complexity}~\cite{HKL98}. We refer the reader to \cite{Lou94} or \cite{HKL98} for the definition. An equivalent definition is: $E$ is \textbf{potentially} $\mathbf{\Gamma}$, for a point-class $\mathbf{\Gamma}$, if $E$ is Borel reducible to some equivalence relation $F$, where $F$ is in $\mathbf{\Gamma}$.

Hjorth, Kechris, and Louveau~\cite{HKL98} proved that the optimal potential complexity of $=^{+n}$ is precisely $\mathbf{\Pi}^0_{2+n}$. In fact, they proved that among $S_\infty$-actions $=^{+n}$ is a maximal equivalence relation with this potential complexity. Moreover, they extended these results for the transfinite jumps and completely classified the possible potential complexities of Borel equivalence relations induced by an $S_\infty$ action.

Here we simply note that the equivalence relations $F_n$, defined to optimize Baire-category considerations, naturally have the optimal potential complexity.
\begin{prop}
The relation $F_n$ is $\mathbf{\Pi}^0_{2+n}$ as a subset of $X_n\times X_n$.
\end{prop}
\begin{proof}
For $F_1$, which is $=^+$, a direct computation shows that it is $\mathbf{\Pi}^0_3$.

We define relations $Q_n$ on $X_{n}\times2^\N$, for $0\leq n<\omega$, so that
\begin{itemize}
    \item $Q_n$ is $\mathbf{\Pi}_{n+1}$, and
    \item for $1\leq n$, for $x,y\in X_{n+1}$, $x\mathrel{F_{n+1}}{y}$ if and only if $u^{n+1}_n(x)\mathrel{F_n}u^{n+1}_n(y)$ and
    \begin{equation*}
        \forall n_1\exists n_2\forall l_1,l_2[(u^{n+1}_{n-1}(x),x(n-1)(l_1))\mathrel{Q_{n-1}}(u^{n+1}_{n-1}(y),y(n-1)(l_2))\rightarrow x(n)(n_1)(l_1)=y(n)(n_2)(l_2)].
    \end{equation*}
\end{itemize}
Assuming this, for $1\leq n$, as $Q_{n-1}$ is $\mathbf{\Pi}_{n}$, the expression in the square brackets is $\mathbf{\Sigma}_{n}$, which shows that $F_{n+1}$ is  $\mathbf{\Pi}^0_{n+3}$, as required.

For $n=0$, define $Q_0$ as equality on $2^\N$. Note that we identify $X_0$ as a space with 1 member, and so we identify $X_0 \times 2^\N$ with $2^\N$. For $n\geq 1$, given $(x,v)\in X_{n}\times 2^\N$, recall the definition of $A^x_{n}=\set{a^{x,l}_{n-1}}{l\in\N}$. Define $a^{x,v}_{n}=\set{a^{x,t}_{n-1}}{v(t)=1}\subset A^x_{n}$.
Given $(x,v),(y,w)\in X_{n}\times 2^\N$, define
\begin{equation*}
    (x,v)\mathrel{Q_n}(y,w)\iff A^x_{n}=A^y_{n}\wedge a^{x,v}_{n}\mathrel{=} a^{y,w}_{n}.
\end{equation*}
The relation $(x,v)\mathrel{Q_{n+1}}(y,w)$ is true if and only if
\begin{equation*}
    (x\mathrel{F_{n}}{y})\wedge \forall l_1,l_2[(u^n_{n-1}(x),x(n-1)(l_1))\mathrel{Q_{n}}(u^n_{n-1}(y),y(n-1)(l_2))\rightarrow v(l_1)=w(l_2)].
\end{equation*}
Inductively, $F_{n}$ is $\mathbf{\Pi}^0_{n+2}$ and $Q_{n}$ is $\mathbf{\Pi}^0_{n+1}$. We conclude that $Q_{n+1}$ is $\mathbf{\Pi}^0_{n+2}$.

\end{proof}

\section{Complexity on comeager sets: some ideas and some obstacles}\label{Section: obstacles first round}
One obstacle towards the $n>1$ case was already encountered. The natural topology coming from the Friedman-Stanley jump operation does not work (see ~\ref{claim: =++ reduction on comeager}), and we therefore had to find the ``correct'' presentation of these equivalence relations, as in Section~\ref{Section: def of Fn into}. 

Let us focus on a corollary of the main theorem, that the equivalence relations $F_n$ retain their complexity on comeager sets (see Section~\ref{subsection : spectrum of the meager ideal}). In this section we sketch some ideas behind the proof for $F_1$, and explain why a different type of construction is necessary to deal with $F_n$ for $n>1$.

\subsection{The case $n=1$}\label{subsection : warmup n=1 case}
The fact that $F_1$ (which is $=^+$) retains its complexity on comeager sets was proven in \cite{ksz}. Given a comeager set $C\subset (2^\N)^\N$, let $C^\ast$ be its Vaught transform (see~\cite[3.2.2]{Gao09}), $C^\ast=\set{a\in (2^\N)^\N}{(\forall^\ast g\in S_\infty)\, g\cdot a \in C}$.

Fix a map $g\colon (2^\N)\to (2^\N)^\N$. Define $f_0\colon (2^\N)^\N\to (2^\N)^{\N\times \N}$ by \[f_0(x)(k,l)=g(x(k))(l).\] 
Fix a bijection $e\colon \N\to \N\times \N$. This extends naturally to a homeomorphism $\hat{e}\colon (2^\N)^{\color{blue}\N\times\N}\to (2^\N)^{\color{blue}\N}$. Let $C\colon (2^\N)^\N\to (2^\N)^\N$ be a Borel map so that for $x\in (2^\N)^\N$, if the set enumerated by $x$ is finite, then $C(x)=x$, and if the set enumerated by $x$ is infinite, then $C(x)$ is an injective enumeration of the same set. Define \[f=C\circ \hat{e}\circ f_0.\] For a ``sufficiently generic'' choice of map $g$, it can be verified that $f(x)\in C^\ast$ for all $x\in (2^\N)^\N$. It follows that there is a Borel map $\rho\colon (2^\N)^\N \to S_\infty$ so that $\rho(x)\cdot f(x) \in C$. Finally, the map $x\mapsto \rho(x)\cdot f(x)$ is a reduction of $F_1$ to $F_1\restriction C$.


\subsection{The case $n > 1$}
Below we explain why a direct generalization of the construction in Section~\ref{subsection : warmup n=1 case}, to construct a map reducing $F_2$ to some comeager subset of $F_2$, does not work. Fix $(x,y)\in (2^\N)^\N \times (2^\N)^\N$ in the domain of $F_2$. We would want to define $f(x,y)$ to be of the form $(u,v)$ so that $(u,v)$ is ``sufficiently generic'', in the sense that it lands in the Vaught transform of some comeager set. 

We can start by defining $u$ from $x$ as before, so that $u$ is ``sufficiently generic''. We may hope to define $v$ from $y$ in the same way, so that $v$ is also ``sufficiently generic''. 
The problem can be seen from the group action presentation in Section~\ref{subsection : orbit ER presentation}. At the second level, we have the usual permutation action of $S_\infty$ (the action $a$), but also another copy of $S_\infty$ acting ``from behind'' via the action $b$. The construction in Section~\ref{subsection : warmup n=1 case}, which is invariant under the action $a$, is \emph{not} invariant under the action $b$, and therefore the resulting map will not respect $F_2$.

This difficulty can also be seen from a set theoretic perspective. Let $A = A^x$ and $B = A^u$, their corresponding classifying $F_1$-invariants. We may want to replace the space $\mathcal{P}_{\aleph_0}(\N)$ (which is identified with $2^\N$) with the space $\mathcal{P}_{\aleph_0}(A)$. Now we may hope to follow the construction of Section~\ref{subsection : warmup n=1 case} to find a ``definable'' map taking some $Y\in \mathcal{P}_{\aleph_0}(A)$ to a ``sufficiently generic'' member of $\mathcal{P}_{\aleph_0}(B)$. (The quotation marks are intended to mean that once translated in a reasonable way to a map defined on our Polish space $X_2$, it will be Borel definable, and land in some comeager set.) This construction should be done independently of the enumerations of $A$ and $Y$, for the resulting map to be a homomorphism $F_2\to_B F_2$.
This hope is immediately crushed. Such constructions are common with $A=\omega$, or more generally an ordinal, but impossible for higher rank sets.

The point of this discussion is to mention that our construction of $v$, towards $f(x,y)=(u,v)$, has to rely on the enumerations coming from $x$, while ultimately being independent of those, up to $F_2$-equivalence. It cannot be done by a direct iteration of the previous construction. The main new construction, which deals with the $n>1$ case, is presented in Section~\ref{section: the main construction}, Definition~\ref{defn: main construction}.

We also present in Section~\ref{section: variation of n=1 case construction} a variation of the above sketched construction for the $n=1$ case. This variation is needed simply to ``align'' the two constructions, as in Section~\ref{subsec: defn of homomorphism f}.

\section{Permutations}\label{Section: permutations}

In various points below we will want a member of some product space, constructed in a specific way, to land in some comeager set. As in Section~\ref{subsection : warmup n=1 case} we will be able to guarantee this only after applying a group action. The following lemma will be used to deal with the construction for the $n>1$ case.

Let $S, X, Y_1, \dots , Y_k$ be infinite sets, considered as discrete metric spaces. Consider the space $(2^S)^X\times(2^X)^{Y_1}\times\dots\times (2^X)^{Y_k}$ with the product topology. Consider the natural diagonal action of $\mathrm{Sym}(X)$ on $(2^S)^X\times(2^X)^{Y_1}\times\dots\times (2^X)^{Y_k}$, acting on all copies of $X$ simultaneously.
For $i=1,\dots,k$, consider the natural action of and $\mathrm{Sym}(Y_i)$ on $(2^S)^X\times(2^X)^{Y_1}\times\dots\times (2^X)^{Y_k}$. These actions commute, leading to an action 
\begin{equation*}
    \mathrm{Sym}(X)\times\mathrm{Sym}(Y_1)\times\dots\times\mathrm{Sym}(Y_k)\curvearrowright (2^S)^X\times(2^X)^{Y_1}\times\dots\times (2^X)^{Y_k}.
\end{equation*}
We consider each $\mathrm{Sym}(Y_i)$, and $\mathrm{Sym}(X)$, as a topological group with the point-wise convergence topology, and $\mathrm{Sym}(X)\times\mathrm{Sym}(Y_1)\times\dots\times\mathrm{Sym}(Y_k)$ with the product topology.

\begin{lemma}\label{lemma: main technical sufficient condition for genericity}
Let $D\subset (2^S)^X\times(2^X)^{Y_1}\times\dots\times (2^X)^{Y_k}$ be dense open.
Fix $(\zeta,\xi_1,\dots,\xi_k)\in (2^S)^X\times(2^X)^{Y_1}\times\dots\times (2^X)^{Y_k}$ satisfying the following assumptions:
\begin{enumerate}
    \item For any finite permutation $\pi$ of $X$, the set 
    \begin{equation*}
        D_{\pi\cdot\zeta}=\set{(\delta_1,\dots\delta_k)}{(\pi\cdot\zeta,\delta_1,\dots,\delta_k)\in D}
    \end{equation*} 
    is dense in $(2^X)^{Y_1}\times\dots\times (2^X)^{Y_k}$.
    \item 
\begin{enumerate}[(a)]
    \item For any finite partial function $\tau\colon X\to\{0,1\}$ and any $i\in\{1,\dots,k\}$, there are infinitely many $y\in Y_i$ so that $\xi_i(y)(\ndash)$, considered as a function $X\to \{0,1\}$, extends $\tau$. 
    \item Given finite partial functions $\tau\colon S\to\{0,1\}$, $\tau_i\colon Y_i\to\{0,1\}$, $i=1,\dots,k$, there are infinitely many $x\in X$ so that $\xi_i(\ndash)(x)$, considered as a function $Y_i\to\{0,1\}$, extends $\tau_i$, for every $i=1,\dots,k$, and $\zeta(x)(\ndash)$, considered as a function $S\to\{0,1\}$, extends $\tau$.
\end{enumerate}
\end{enumerate}
 Then the set
\begin{equation*}
    G=\set{(g,g_1,\dots,g_k)\in \mathrm{Sym}(X)\times\mathrm{Sym}(Y_1)\times\dots\times\mathrm{Sym}(Y_k)}{(g,g_1,\dots,g_k)\cdot (\zeta,\xi_1,\dots,\xi_k))\in D}
\end{equation*}
is dense open in $\mathrm{Sym}(X)\times\mathrm{Sym}(Y_1)\times\dots\times\mathrm{Sym}(Y_k)$.
In particular, if $D$ is assumed to be comeager, then $G$ is concluded to be comeager.
\end{lemma}
\begin{proof}
First, since the map $\mathrm{Sym}(X)\times\mathrm{Sym}(Y_1)\times\dots\times\mathrm{Sym}(Y_k) \to (2^S)^X\times(2^X)^{Y_1}\times\dots\times (2^X)^{Y_k}$, $(g,g_1,\dots,g_k)\mapsto (g,g_1,\dots,g_k)\cdot(\zeta,\xi_1,\dots,\xi_k)$, is continuous, then $G$ is open as the pre-image of $D$. 

Next we prove that $G$ is dense. Fix finite partial permutations $\pi,\pi_1,\dots,\pi_k$ of $X,Y_1,\dots,Y_k$ respectively. We need to find an extension of these in $G$. Let $\Bar{X},\Bar{Y_1},\dots,\Bar{Y_k}$ be the finite supports of $\pi,\pi_1,\dots,\pi_k$, respectively.

By assumption (1), $D_{\pi\cdot\zeta}$ is dense. Fix $(\delta_1,\dots,\delta_k)\in D_{\pi\cdot\zeta}$ which agree with $(\pi_1\cdot\xi_1,\dots,\pi_k\cdot\xi_k)$ on $(2^{\bar{X}})^{\bar{Y}_1}\times\dots\times (2^{\bar{X}})^{\bar{Y}_k}$. Since $D$ is open, we may find finite $\hat{X},\hat{Y_1},\dots,\hat{Y_k}$, extending $\Bar{X},\Bar{Y_1},\dots,\Bar{Y_k}$, and a finite set $\hat{S}$, so that if $(\zeta',\xi'_1,\dots,\xi'_k)$ agree with $(\pi\cdot\zeta,\delta_1,\dots,\delta_k)$ on $(2^{\hat{S}})^{\hat{X}}\times(2^{\hat{X}})^{\hat{Y}_1}\times\dots\times (2^{\hat{X}})^{\hat{Y}_k}$, then $(\zeta',\xi'_1,\dots,\xi'_k)\in D$. It remains to find $(g,g_1,\dots,g_k)$, extending $\pi,\pi_1,\dots,\pi_k$, so that $(g,g_1,\dots,g_k)\cdot (\zeta,\xi_1,\dots,\xi_k)$ and $(\pi\cdot\zeta,\delta_1,\dots,\delta_k)$ agree on $(2^{\hat{S}})^{\hat{X}}\times(2^{\hat{X}})^{\hat{Y}_1}\times\dots\times (2^{\hat{X}})^{\hat{Y}_k}$.

For each $i=1,\dots,k$, for each $y\in \hat{Y}_i\setminus \bar{Y}_i$, consider the function $\tau_y\colon \bar{X}\to\{0,1\}$, $\tau_y(x)=\delta_i(y)(x)$. By assumption (2)(a), there are infinitely many $y'\in Y_i$ so that $\xi_i(y')(\ndash)$ and $\tau_y(\ndash)$ agree on $\bar{X}$. It follows that there is a finite permutation $g_i$ extending $\pi_i$ so that $(g_i\cdot \xi_i)(y)(\ndash)$ and $\delta_i(y)(\ndash)$ agree on $\bar{X}$, for all $y\in \hat{Y}_i\setminus\bar{Y_i}$. It follows that $(\pi,g_1,\dots,g_k)\cdot (\zeta,\xi_1,\dots,\xi_k)$ and $(\pi\cdot\zeta,\delta_1,\dots,\delta_k)$ agree on $(2^{\hat{S}})^{\hat{X}}\times(2^{\bar{X}})^{\hat{Y}_1}\times\dots\times (2^{\bar{X}})^{\hat{Y}_k}$.

Note that conditions (2)(a) and (2)(b) of the lemma are invariant under the group action. We apply condition (2)(b) to $(\pi,g_1,\dots,g_k)\cdot (\zeta,\xi_1,\dots,\xi_k)$. We may write $(\pi,g_1,\dots,g_k)\cdot (\zeta,\xi_1,\dots,\xi_k)$ as $(\pi\cdot\zeta,(\pi,g_1)\cdot\xi_1,\dots,(\pi,g_k)\cdot\xi_k)$, where $(\pi,g_i)\cdot\xi_i$ refers to the action $\mathrm{Sym}(X)\times\mathrm{Sym(Y_i)}\curvearrowright (2^S)^X\times (2^X)^{Y_i}$.

For each $i=1,\dots,k$, for each $x\in \hat{X}\setminus \bar{X}$, consider the function $\tau^x_i\colon \hat{Y}_i\to\{0,1\}$, $\tau^x_i(y)=\delta_i(y)(x)$. Define also $\tau^x\colon \hat{S}\to\{0,1\}$ by $\tau^x(s)=(\pi\cdot\zeta)(x)(s)$. For each $x\in \hat{X}\setminus\bar{X}$ there are infinitely many $x'\in X$ so that $((\pi,g_i)\cdot\xi_i)(\ndash)(x')$ extends $\tau^{x}_i$ and $(\pi\cdot\zeta)(x')(\ndash)$ extends $\tau^{x}$.
It follows that there is a finite permutation $g$ extending $\pi$ so that $((g,g_i)\cdot\xi_i)(\ndash)(x)$ and $\delta_i(\ndash)(x)$ agree on $\hat{Y}_i$, for all $i$ and any $x\in\hat{X}\setminus \bar{X}$, and $(g\cdot\zeta)(\ndash)(x)$ agrees with $\pi\cdot\zeta(\ndash)(x)$ on $\hat{S}$ for any $x\in\hat{X}\setminus \bar{X}$.
We conclude that $(g,g_1,\dots,g_k)\cdot(\zeta,\xi_1,\dots,\xi_k)$ and $(\pi\cdot\zeta,\delta_1,\dots,\delta_k)$ agree on $(2^{\hat{S}})^{\hat{X}}\times(2^{\hat{X}})^{\hat{Y}_1}\times\dots\times (2^{\hat{X}})^{\hat{Y}_k}$, as required.
\end{proof}

\subsection{The $n=1$ case}
When dealing with the first coordinate of $F_n$ we only have the action $a\colon S_\infty{\color{blue}\curvearrowright}(2^\N)^{\color{blue}\N}$. More generally, we deal with product actions of the form
\begin{equation*}
    \mathrm{Sym}(Y_1)\times\dots\times\mathrm{Sym}(Y_k)\curvearrowright (2^X)^{Y_1}\times\dots\times (2^X)^{Y_k}.
\end{equation*}
\begin{lemma}\label{lemma: main technical sufficient condition for genericity n=1 case}
Fix a dense open $D\subset (2^X)^{Y_1}\times\dots\times (2^X)^{Y_k}$ and countable infinite sets $M_1,\dots,M_k$ . Then there is a dense open set $D'\subset (2^X)^{M_1}\times\dots\times(2^X)^{M_k}$ so that for any $\zeta\in (2^X)^{Y_1}\times\dots\times (2^X)^{Y_k}$, if $\zeta$ satisfies the following property: given any finite partial injective functions $\tau_i\colon M_i\to Y_i$, $i=1,\dots,k$, there are extensions $\alpha_i\colon M_i\to Y_i$ so that $\left( \zeta(i)\circ \alpha_i : i<k \right)$, a member of the space $(2^X)^{M_1}\times\dots\times (2^X)^{M_k}$, is in $D'$, then the set
\begin{equation*}    G=\set{(g_1,\dots,g_k)\in \mathrm{Sym}(Y_1)\times\dots\times\mathrm{Sym}(Y_k)}{(g_1,\dots,g_k)\cdot \zeta\in D}
\end{equation*}
is dense open in $\mathrm{Sym}(Y_1)\times\dots\times\mathrm{Sym}(Y_k)$.
In particular, if $D$ is assumed to be comeager, then there is a comeager $D'$ so that $G$ is concluded to be comeager.
\end{lemma}
\begin{proof}
First, since the map $\mathrm{Sym}(Y_1)\times\dots\times\mathrm{Sym}(Y_k) \to (2^X)^{Y_1}\times\dots\times (2^X)^{Y_k}$, $(g_1,\dots,g_k)\mapsto (g_1,\dots,g_k)\cdot \zeta$, is continuous, then $G$ is open as the pre-image of $D$. 

Next, we describe the set $D'$ so that if $\zeta$ satisfies the assumption in the lemma, then the set $G$ is dense.  First assume that $M_i = Y_i$ and take $D'$ to be $D$. The assumption tells us that for any finite partial permutations $\tau_i\colon M_i \to Y_i$ there are extensions to total injective maps $\alpha_i \colon Y_i \to Y_i$ so that $\seqq{\zeta(i)\circ\alpha_i}{i<k}\in D$. If $\alpha_i$ were all bijections, so in $\mathrm{Sym}(Y_i)$, we would be done. Nevertheless, since $D$ is open, we may find bijections $\sigma_i$ of $Y_i$ which extend $\tau_i$ and are sufficiently close to $\alpha_i$ so that $\seqq{\zeta(i)\circ\sigma_i}{i<k}\in D$ as well. Finally, for any infinite countable $M_1,\dots,M_k$, we may fix bijections $M_i \to Y_i$, resulting in a homeomorphism between $(2^X)^{Y_1}\times\dots\times (2^X)^{Y_k}$ and $(2^X)^{M_1}\times\dots\times (2^X)^{M_k}$. We let $D'$ be the image of $D$. 
\end{proof}

\section{The main construction}\label{section: the main construction}
\begin{defn}\label{defn: main construction}
Fix a function 
\begin{equation*}
    \alpha\colon (2^{<\N})^{<\N}\times\N\times\N\to 2
\end{equation*} (which will be chosen to be ``sufficiently generic''). Define
\begin{equation*}
\beta\colon (2^\N)^{<\N}\to (2^{\N^{<\N}\times\N})^\N\textrm{ by}    
\end{equation*}
    \begin{equation*}
        \beta(x_1,\dots,x_l)(m)(t,k)=\alpha(x_1\circ t,\dots, x_l\circ t, k, m)\textrm{, for }t\in\N^{<\N}; k,m\in\N.
    \end{equation*}
 Each $x_i$ is considered a function $\N\to 2$, and so $x_i\circ t$ is a member of $2^{<\N}$. Define
 \begin{equation*}
 \gamma\colon (2^\N)^\N\to (2^{\N^{<\N}\times\N})^{\N^{<\N}\times\N}\textrm{ by}    
 \end{equation*}
 \begin{equation*}
     \gamma(x)(t,k)=\beta(x\circ t)(k).
 \end{equation*}
\end{defn}

\begin{remark}\label{remark: gamma}
    \begin{itemize}
        \item The function $\gamma$ is continuous.
        \item  The function $\gamma$ is a homomorphism from the orbit equivalence relations
        \begin{equation*}
            \mathrm{Sym}(\N){\color{red}\curvearrowright}(2^{\color{red}\N})^\N\textrm{ to } \mathrm{Sym}(\N^{<\N}\times \N){\color{red}\curvearrowright}(2^{\color{red}\N^{<\N}\times\N})^{\N^{<\N}\times\N}\textrm{, and }
        \end{equation*}
        \begin{equation*}
            \mathrm{Sym}(\N){\color{blue}\curvearrowright}(2^\N)^{\color{blue}\N}\textrm{ to } \mathrm{Sym}(\N^{<\N}\times\N){\color{blue}\curvearrowright}(2^{\N^{<\N}\times\N})^{\color{blue}\N^{<\N}\times\N}
        \end{equation*}
        \item The definition of $\gamma$ relies on a choice of $\alpha$. We will show that there is some $\alpha$ for which $\gamma$ satisfies the properties which we need. This will happen for $\alpha$ chosen generically, with respect to the product topology $2^{(2^{<\N})^{<\N}\times\N\times\N}$.        
    \end{itemize}
\end{remark}

\begin{Notation}\label{notation: N}
Let $N=\N^{<\N}\times \N$. 
\end{Notation}

It will be convenient, to utilize the construction above, to work with the space $(2^N)^N$ instead of $(2^\N)^\N$.
To illustrate the construction, consider the following lemma. 

Say that $a\in(2^\N)^\N$ is \textbf{injective} if it is a sequence of distinct reals: $a(i)\neq a(j)$ for $i\neq j$. Say that $a\in (2^\N)^\N$ is \textbf{separated} if any distinct $n,k\in\N$ are separated by one of the members of $a$: there is some $i\in\N$ so that $a(i)(n) \neq a(i)(k)$.

\begin{lemma}\label{lemma: examples of main construction}
 For a generic $\alpha$ the following holds. Suppose $x\in (2^\N)^\N$ is separated and $y\in (2^\N)^\N$ is injective and separated. Then the pair $(\zeta,\xi_1) = (\gamma(x),\gamma(y))\in (2^S)^X\times (2^X)^{Y_1}$, where $S=N$, $X = N$, and $Y_1=N$, satisfies the assumptions (2) in Lemma~\ref{lemma: main technical sufficient condition for genericity}. That is:
        \begin{itemize}
            \item for any finite $\tau\colon N\to\{0,1\}$ there are infinitely many $(t,k)\in N$ so that $\gamma(y)(t,k)(\ndash)$ extends $\tau$, and
            \item for any finite $\tau_1\colon N\to \{0,1\}$ and $\tau_2\colon N\to\{0,1\}$ there are infinitely many $(s,d)\in N$ so that $\gamma(y)(\ndash)(s,d)$ extends $\tau_1$ and $\gamma(x)(s,d)(\ndash)$ extends $\tau_2$.
        \end{itemize}
\end{lemma}
\begin{proof}
Consider the second bullet, which corresponds to condition (b) Lemma~\ref{lemma: main technical sufficient condition for genericity}. Fix finite partial function $\tau_1\colon Y\to \{0,1\}$, $\tau_2\colon S\to\{0,1\}$. Recall the definitions. Fix $d,k\in \N$ and $s,t\in \N^{<\N}$, $s=\left(s_1,\dots,s_m\right)$, $t=\left(t_1,\dots,t_l\right)$ for some $m,l$.     We abbreviate
    \begin{equation*}
        \begin{split}
     x[s,t] & = (x(s_1)\circ t,\dots, x(s_m)\circ t)\in (2^{<\N})^m, \\
     y[t,s] & = (y(t_1)\circ s,\dots,y(t_l)\circ s) \in (2^{<\N})^l. 
        \end{split}
    \end{equation*}
    Then 
    \begin{equation*}
    \begin{split}
        \gamma(x)(s,d)(t,k) & = \alpha(x[s,t],k,d),\\
        \gamma(y)(t,k)(s,d) & =\alpha(y[t,s],d,k).
    \end{split}
    \end{equation*}
    So we need to find infinitely many $(s,d)$ for which 
    \begin{equation*} \hspace{1cm}(\star)\hspace{-2cm}
    \begin{split}
       \hspace{1cm}    \alpha(x[s,t],k,d) & =\tau_2(t,k)\textrm{, for all } (t,k)\textrm{ in the domain of }\tau_2.\\
       \alpha(y[t,s],d,k) & =\tau_1(t,k)\textrm{, for all } (t,k)\textrm{ in the domain of }\tau_1.
    \end{split}
    \end{equation*}
Since $x$ is separated, given $t_1\neq t_2\in \N^{<\N}$ there is some $i\in\N$ so that $x(i)\circ t_1 \neq x(i)\circ t_2$. 
Construct a sequence $s^*\in \N^{<\N}$ so that for any two distinct $t_1,t_2$ in the domain of $\tau_2$ there is some $i$ so that $x(s^*_i)\circ t_1 \neq x(s^*_i)\circ t_2$. Note that if $s$ is a sequence which contains $s^*$ then $x[s,t_1]\neq x[s,t_2]$, for any distinct $t_1,t_2$ in the domain of $\tau_2$.

Since $y$ is injective, for any $t_1\neq t_2$ the finite sequences of reals $y\circ t_1$ and $y\circ s_1$ are not equal. Then for any $s$ whose range contains a long enough initial segment of $\N$ (depending on $t_1,t_2$), $y[t_1,s]\neq y[t_2,s]$. We may find some $s^{**}$ so that for any $s$ which contains $s^{**}$, for any distinct $t_1,t_2$ in the domain of $\tau_1$, $y[t_1,s]\neq y[t_2,s]$.

Fix $s$ which contains both $s^*$ and $s^{**}$. Then for any $(t_1,k_1),(t_2,k_2)$ from the domain of $\tau_2$ or $\tau_1$, if the tuples $(t_1,k_1), (t_2,k_2)$ are distinct, then 
\begin{itemize}
    \item the tuples $(x[s,t_1],k_1), (x[s,t_2],k_2)$ are distinct, and
    \item the tuples $(y[t_1,s],k_1),(y[t_2,s],k_2)$ are distinct.
\end{itemize} 
We claim that, for a generic choice of $\alpha$, there are infinitely many $d\in\N$ for which $(\star)$ holds with $(s,d)$. It suffice to prove that for any finite number $M$ there is a dense open set of $\alpha\in 2^{(2^{<\N})^{<\N}\times\N\times\N}$ for which there are at least $M$ many values of $d$ so that $(\star)$ holds for $(s,d)$. Indeed, given any finite amount of information about $\alpha$, for a large enough $d$ we have that
\begin{itemize}
    \item $\alpha(x[s,t],k,d)$, $\alpha(y[s,t],d,k)$ is not yet defined, 
    \item and the tuples $(y[t,s],d,k), (x[s,t],k,d)$ respectively are distinct for distinct $(t,k)$ in the domains of either $\tau_1,\tau_2$.
\end{itemize}
We may therefore extend $\alpha$ by adding arbitrarily many values of $d$ for which $(\star)$ holds for $(s,d)$. Note that there are only countably many dense open sets for $\alpha\in 2^{(2^{<\N})^{<\N}\times\N\times\N}$ involved in the argument above, independently of $x,y$.

Finally, note that the property for $\gamma(y)$ in the first bullet is the same as the property for $\gamma(x)$ in the second bullet. Since $y$ is separated, which is the only assumption we made on $x$ in the proof, we conclude the first bullet as well.


\end{proof}

\subsection{A variation for the $n=1$ construction}\label{section: variation of n=1 case construction}
We will also modify the construction in the $n=1$ case, as outlined in Section~\ref{subsection : warmup n=1 case}, to be compatible with the above.
Fix a continuous function $\beta_0\colon (2^\N)^{<\N}\to (2^{N})^\N$ (which will be chosen to be ``sufficiently generic'').
Define $\gamma_0\colon (2^\N)^\N\to (2^{N})^{N}$ by
    \begin{equation*}
        \gamma_0(x)(t,k)=\beta_0(x\circ t)(k)\textrm{, for }t\in \N^{<\N}, k\in\N.
    \end{equation*}

\begin{remark}\label{remark: gamma0}
    \begin{itemize}
        \item The function $\gamma_0$ is continuous.
        \item The function $\gamma_0$ is a homomorphism from the orbit equivalent relation 
        \begin{equation*}
            \mathrm{Sym}(\N){\color{blue}\curvearrowright}(2^\N)^{\color{blue}\N} \textrm{ to } \mathrm{Sym}(\N^{<\N}\times\N){\color{blue}\curvearrowright}(2^{\N^{<\N}\times\N})^{\color{blue}\N^{<\N}\times\N}.
        \end{equation*}
    \end{itemize}
\end{remark}

The function $\beta_0$ is constructed as follows. Fix comeager sets $C_n\subset ((2^N)^\N)^n$. Fix a continuous $\alpha_0\colon 2^\N\to (2^N)^\N$ with the property that for any pairwise distinct $x_1,\dots,x_n\in 2^\N$, $(\alpha_0(x_1),\dots,\alpha_0(x_n))\in C_n$ (see~\cite[Theorem~19.1]{Kechris-DST-1995}). Define $\beta_0 \colon (2^\N)^{<\N}\to (2^N)^\N$ by $\beta_0 = \alpha_0\circ \iota$, where $\iota\colon (2^\N)^{<\N}\hookrightarrow 2^\N$ is a continuous injective map.
\begin{Notation}
    We will say ``for almost any $\beta_0$'' or ``for a generic $\beta_0$'' to mean $\beta_0$ constructed with an appropriate choice of the comeager sets $C_n$.
\end{Notation}

As in Lemma~\ref{lemma: examples of main construction}, consider the following example.

\begin{lemma}\label{lemma: examples of main construction n=1 case}
Suppose $x\in (2^\N)^\N$ is injective. Then for any comeager $C\subset (2^N)^N$, for a generic $\beta_0$, $\gamma_0(x)\in (2^N)^N$ satisfies the assumption in Lemma~\ref{lemma: main technical sufficient condition for genericity n=1 case}, with $X=N$, $Y_1=N$, and $M_1=\N$. That is, for $C'\subset (2^N)^\N$ given by Lemma~\ref{lemma: main technical sufficient condition for genericity n=1 case}, for any finite partial injective $\tau\colon \N\to N$ there is an extension $\sigma\colon\N\to N$ of $\tau$ so that $\gamma_0(x)\circ\sigma \in C'$.

\end{lemma}
\begin{proof}
We describe how to choose the sets $C_n\subset ((2^\N)^N)^n$ in the definition of $\beta_0$. 
Fix a bijection $e\colon \N\to \N^{<\N}$ and let $N_n = \set{e(i)}{i < n}$, for $0<n$. For each $n\in\N$, fix a bijection $s_n \colon \N\to N_n \times \N$, inducing a homeomorphism $\hat{s}_n \colon (2^N)^{N_n\times\N}\to (2^N)^\N$. In particular, $\hat{s}_n^{-1}D'$ is dense open in $(2^N)^{N_n\times \N}$, which we identify with $((2^N)^\N)^{N_n}$, and in turn with $((2^N)^\N)^n$. Choose comeager $C_n$ which is a subset of $\hat{s}^{-1} D'$ for any bijection $s\colon \N\to N_n\times\N$ which agrees with $s_n$ on all but finitely many values.
Since $x$ is injective, $x\circ t_1 \neq x\circ t_2$ for distinct $t_1,t_2\in \N^{<\N}$. By the choice of $\beta_0$, $\seq{\beta_0(x\circ e(0)), \dots, \beta_0(x\circ e(n-1))}\in C_n$ for all $0<n$. 

Let $\tau\colon \N\to N$ be a finite partial injection. Recall that $N=\N^{<\N}\times \N$. Choose $n$ so that the image of $\tau$ is included in $N_n \times \N$. Let $s \colon \N\to N_n\times \N$ be a bijection which extends $\tau$ and agrees with $s_n$ on all but finitely many values. Note that $\hat{s}^{-1}(\gamma_0(x)\circ s)\in (2^N)^{N_n\times \N}$, which is identified with $\seq{\beta_0(x\circ e(0), \dots, \beta_0(x\circ e(n-1)))}$, is in $C_n$, and therefore $\gamma_0\circ s \in D'$, as required.   
\end{proof}

We will also need the following variation of Lemma~\ref{lemma: examples of main construction}.

\begin{lemma}\label{lemma: example of main construction - combination}
    For generic $\alpha$ and $\beta_0$ the following holds. Suppose $x\in (2^\N)^\N$ is injective, and $y\in (2^\N)^\N$ is injective and separated. The pair $(\zeta,\xi_1) = (\gamma_0(x),\gamma(y))\in (2^S)^X\times (2^X)^{Y_1}$, where $S=N$, $X = N$, and $Y_1=N$, satisfies the assumptions (2) in Lemma~\ref{lemma: main technical sufficient condition for genericity}. That is:
        \begin{itemize}
            \item for any finite $\tau\colon N\to\{0,1\}$ there are infinitely many $(t,k)\in N$ so that $\gamma(y)(t,k)(\ndash)$ extends $\tau$, and
            \item for any finite $\tau_1\colon N\to \{0,1\}$ and $\tau_2\colon N\to\{0,1\}$ there are infinitely many $(s,d)\in N$ so that $\gamma(y)(\ndash)(s,d)$ extends $\tau_1$ and $\gamma(x)(s,d)(\ndash)$ extends $\tau_2$.
        \end{itemize}
\end{lemma}
Most aspects of the proof are similar to Lemma~\ref{lemma: examples of main construction}

\subsection{A change of countable base set}\label{section: change of countable set}
Recall Notation~\ref{notation: N}: $N=\N^{<\N}\times \N$. 

\begin{defn}[Definition of $F_n^\ast$ on $X_n^\ast$]
Consider the space $((2^N)^N)^\omega$, which is naturally homeomorphic to $((2^\N)^\N)^\omega$. We may define $F_n^\ast$ on $X_n^\ast\subset ((2^N)^N)^n$, as $F_n$ was defined on $X_n\subset ((2^\N)^\N)^n$ in Section~\ref{Section: def of Fn into}. See Section~\ref{section: appendix} for more details.
\end{defn}
The equivalence relation $F_n^\ast$ is isomorphic to $F_n$ via a homeomorphism of their domains. To prove that $F_n$ retains its complexity on comeager sets, it suffices to prove that $F_n^\ast$ retains its complexity on comeager sets.

\subsection{The homomorphism $f\colon F_n \to F_n^\ast$}\label{subsec: defn of homomorphism f}
\begin{defn}\label{defn: the homomorphism f}
    Define $f\colon ((2^\N)^\N)^n\to ((2^N)^N)^n$ by $f = \gamma_0 \times \gamma^{n\setminus \{0\}}$. That is,
\begin{equation*}
    f(\seqq{\xi_i}{i<n})=\seqq{\gamma_0(\xi_0), \gamma(\xi_i)}{0<i<n}.
\end{equation*}
\end{defn}
\begin{remark}\label{remark: reduction F_n to F_n generically}
 For a generic choice of $\beta_0$ and $\alpha$,
\begin{itemize}
    \item $f$ sends $X_n$ to $X^\ast_n$, the domains of $F_n$ and $F^\ast_n$ respectively.
    \item $f$ is a continuous reduction of $F_n$ to $F^{\ast}_n$. 
\end{itemize}    
\end{remark}

\section{Complexity on comeager sets}\label{section: complexity on comeager sets}

In this section we prove the following corollary of the main theorem (see Section~\ref{subsection : spectrum of the meager ideal}). This will illustrate the construction of the map $f$, and the techniques developed so far.
\begin{thm}\label{thm: warmup proof retains complexity on comeager sets}
    For $1\leq n \leq \omega$, $F_n$ retains its complexity on comeager sets. That is, ${F_n}\leq_B{F_n\restriction C}$ for any comeager $C\subset X_n$.
\end{thm}

Recall that our main construction is useful to deal with injective and separated sequences (see Lemma~\ref{lemma: examples of main construction}). The following lemma gives a reduction to the case of injective and separated sequences. Its proof is deferred to Section~\ref{section: appendix}.
\begin{lemma}\label{lemma: technical reduction to inj and sep}
    There is a Borel reduction $\psi\colon F_n\to_B F_n$ so that for any $x\in X_n$, if $x\in \mathrm{Image}(\psi)$ then $x(i)\in (2^\N)^\N$ is injective for $i<n$ and separated for $0<i<n$.
\end{lemma}
As discussed in  Section~\ref{section: change of countable set}, it suffices to prove that $F^\ast_n$ retains its complexity on comeager sets. Fix a comeager $C\subset X^\ast_n$. We will find a Borel reduction $f$ from $F_n$ to $F_n^\ast$ sending injective and separated sequences into $C$. This will show that $f\circ\psi$ is a reduction from $F_n$ to $F^\ast_n\restriction C$. 
\begin{prop}
For a generic choice of $\beta_0$ and $\alpha$, if $x\in X_n$ is such that $x(i)\in (2^\N)^\N$ is injective for all $i<n$ and separated for $0<i<n$, then  
\begin{equation*}
    (\forall^\ast g\in \mathrm{Sym}(N)^n)\, g\cdot f(x) \in C.
\end{equation*}
\end{prop}
\begin{proof}
    We prove by induction on $0\leq k<n$ that 
\begin{equation*}
(\star_k)\hspace{1cm}    \forall^\ast h\in \mathrm{Sym}(N)^{k+1}\textrm{ the fiber }C_{h\cdot f(x)\restriction k+1}\textrm{ is comeager in } ((2^N)^N)^{n\setminus k+1}.
\end{equation*}    
First, assume that $(\star_k)$ holds for $0 < k$, $k+1<n$, and prove that $(\star_{k+1})$ holds. 
\begin{claim}
    $\forall^\ast h = (h_0,\dots,h_k)\in \mathrm{Sym}(N)^{k+1}$, the pair 
\begin{equation*}
    ((h_{k-1},h_k)\cdot f(x)(k), (h_k,\mathrm{id})\cdot f(x)(k+1))    \in (2^N)^N\times (2^N)^N
\end{equation*}
satisfies the conditions of Lemma~\ref{lemma: main technical sufficient condition for genericity}, with $S = N$, $X = N$ and $Y_1 = N$, with respect to the comeager set
\begin{equation*}
    \set{(a,b)\in (2^N)^N\times (2^N)^N}{\textrm{the fiber }D=C_{(h\restriction k) \cdot (f(x)\restriction k),a,b}\subset ((2^N)^N)^{n\setminus k+2}\textrm{ is comeager}}.
\end{equation*}
\end{claim}
\begin{proof}[Proof of the claim]
    By assumption $(\star_k)$, $\forall^\ast h=(h_0,\dots,h_k)$, the set $D_{(h_{k-1},h_k)\cdot f(x)(k)}\subset (2^N)^N$ is comeager, for $D$ as above. It follows that $\forall^\ast h=(h_0,\dots,h_k)$, the set $D_{(h_{k-1},\pi\circ h_k)\cdot f(x)(k)}\subset (2^N)^N$ is comeager, for any finite permutation $\pi$ of $N$. This concludes condition (1) in Lemma~\ref{lemma: main technical sufficient condition for genericity}. 
    
Recall that $f(x)(k) = \gamma(x(k))$ and $f(x)(k+1) = \gamma(x(k+1))$. Since $x(k)$ is separated and $x(k+1)$ is injective and separated, it follows from Lemma~\ref{lemma: examples of main construction} that the pair $(f(x)(k), f(x)(k+1))$ satisfies condition (2) in Lemma~\ref{lemma: main technical sufficient condition for genericity}. Note that both parts (a) and (b) of condition (2) are invariant under the group action 
and so true for the pair
\[((h_{k-1},h_k)\cdot f(x)(k), (h_k,\mathrm{id})\cdot f(x)(k+1))\] as well.
This concludes the proof of the claim.
\end{proof}
We conclude from Lemma~\ref{lemma: examples of main construction} that
\begin{equation*}
\begin{split}
    (\forall^\ast (h_0,\dots,h_k)\in\mathrm{Sym}(N)^{k+1})(\forall^\ast(h'_{k},h_{k+1})\in\mathrm{Sym}(N)^2)\textrm{ the fiber } \\ C_{(h_0,\dots,h_{k-1})\cdot f(x)\restriction k,\, (h_{k-1},h'_k\cdot h_k) \cdot f(x)(k),\, (h'_{k} \cdot h_{k}, h_{k+1})\cdot f(x)(k+1)}\textrm{ is comeager in }((2^N)^N)^{n\setminus k+2},
\end{split}
    \end{equation*}
and therefore 
\begin{equation*}
        (\forall^\ast h\in\mathrm{Sym}(N)^{k+2})\textrm{ the fiber } C_{h\cdot f(x)\restriction k+2}\textrm{ is comeager in }((2^N)^N)^{n\setminus k+2}.
\end{equation*}

Finally, we prove the base case of the induction. Let $C_1$ be the set of all $x\in (2^N)^N$ so that the fiber $C_x\subset ((2^N)^N)^{n\setminus 1}$ is comeager. Then $C_1\subset (2^N)^N$ is comeager. 
    Recall that $f(x)(0) = \gamma_0 (x(0))$, and $x(0)\in (2^\N)^\N$ is injective by assumption. By Lemma~
    \ref{lemma: examples of main construction n=1 case} and Lemma~\ref{lemma: main technical sufficient condition for genericity n=1 case}, for a generic choice of $\beta_0$, 
    \begin{equation*}
        (\forall^* h_0\in \mathrm{Sym}(N)) \,h_0\cdot (f(x)(0)) \in C_1.
    \end{equation*}
Let $C_2$ be the set of all $(a,b)\in (2^N)^N\times (2^N)^N$ so that the fiber $C_{a,b}\subset ((2^N)^N)^{n\setminus 2}$ is comeager. We claim that, $\forall^\ast h_0\in \mathrm{Sym}(N)$, the pair
\begin{equation*}
    (h_0\cdot f(x)(0), (h_0,\mathrm{id})\cdot f(x)(1))    
\end{equation*}
 satisfies the conditions of Lemma~\ref{lemma: main technical sufficient condition for genericity}. For part (1), note that as $(C_2)_{h_0\cdot f(x)(0)}$ is comeager $\forall^\ast h_0\in \mathrm{Sym}(N)$, then the same conclusion holds for $\pi\circ h_0$, $\forall^\ast h_0\in\mathrm{Sym}(N)$, for any finite permutation $\pi$ of $N$. For part (2), the pair $(f(x)(0), f(x)(1))$ satisfies the conditions (a) and (b), for a generic choice of $\alpha$ and $\beta_0$, by Lemma~\ref{lemma: example of main construction - combination}, as $x(0)$ is injective and $x(1)$ is injective and separated. As before, since these conditions are invariant under the group action, they are also satisfied by the pair $(h_0\cdot f(x)(0), (h_0,\mathrm{id})\cdot f(x)(1))$. We then conclude from Lemma~\ref{lemma: main technical sufficient condition for genericity} that
\begin{equation*}
   (\forall^\ast h_0\in\mathrm{Sym}(N)) (\forall^*(h'_0,h_1)\in \mathrm{Sym}(N)^2)  (h'_0\cdot h_0\cdot f(x)(0), (h'_0\cdot h_0,h_1)\cdot f(x)(1)) \in C_2, 
\end{equation*}
and therefore
\begin{equation*}
    (\forall^*(h_0,h_1)\in \mathrm{Sym}(N)^2)  (h_0\cdot f(x)(0), (h_0,h_1)\cdot f(x)(1)) \in C_2,
\end{equation*}
concluding that $(\star_1)$ holds.

Finally, we conclude the proof of the proposition. In case $n<\omega$, we conclude at stage $n-1$ that $(\forall^\ast h\in \mathrm{Sym}(N)^n)\left[h\cdot f(x) \in C\right]$, as required. Assume now that $n=\omega$. Let $C_k \subset ((2^N)^N)^{k}$ be the set of all $a\in ((2^N)^N)^k$ for which the fiber $C_a\subset ((2^N)^N)^{\omega\setminus k}$ is comeager. We have that $\forall^\ast h\in\mathrm{Sym}(N)^\omega$, for any $k<\omega$, $h\cdot f(x)\restriction k\in C_k$. It follows that for any sequence of finite permutations $\pi\in\mathrm{Sym}(N)^k$, $\forall^\ast h\in\mathrm{Sym}(N)^\omega$, $(\pi\circ h)\cdot f(x)\restriction k\in C_k$.

Note that the action $a\colon \mathrm{Sym}(N){\color{blue}\curvearrowright} (2^N)^{\color{blue}N}$ is generically ergodic, that is, there is a comeager subset of $(2^N)^N$ in which every orbit is dense. Similarly, the action $\mathrm{Sym}(N)^\omega\curvearrowright ((2^N)^N)^\omega$ is generically ergodic, and for any $k<\omega$ the action $\mathrm{Sym}(N)^{\omega\setminus k}\curvearrowright ((2^N)^N)^{\omega\setminus k}$ is generically ergodic. Note that the orbit of some $b\in ((2^N)^N)^{n\setminus k}$ is dense if and only if for every $k<l<\omega$ the orbit of $b$ restricted to $((2^N)^N)^{[k,l)}$ is dense. By thinning out the comeager set $C$, we may assume that if $\xi\in ((2^N)^N)^\omega$ is such that $\xi\restriction k\in C_k$ for all $k<\omega$, then the orbit of $\xi\restriction[k,\omega)$ is dense for $\mathrm{Sym}(N)^{\omega\setminus k}\curvearrowright ((2^N)^N)^{\omega\setminus k}$, for every $k<\omega$.

\begin{claim}
Assume that $\xi\in ((2^N)^N)^\omega$ is such that $\pi\cdot(\xi\restriction k )\in C_k$, for any sequence of finite permutations $\pi\in \mathrm{Sym}(N)^k$, for all $k<\omega$.
Then
\begin{equation*}
    (\forall^\ast g\in \mathrm{Sym}(N)^\omega) g\cdot \xi \in C.
\end{equation*}
\end{claim}
\begin{proof}
It suffices to prove that for any dense open set $C \subset D \subset ((2^N)^N)^\omega$, the set $\set{g\in\mathrm{Sym}(N)^\omega}{g\cdot \xi \in D}$ is dense open in $\mathrm{Sym}(N)^\omega$. 
Fix a sequence of finite permutations $\pi \in \mathrm{Sym}(N)^k$. We need to find an extension of it, $g\in \mathrm{Sym}(N)^\omega$, so that $g\cdot \xi\in D$. Since $(\pi,\mathrm{id})\cdot(\xi\restriction k+1)\in C_{k+1}$, the fiber $D_{(\pi,\mathrm{id})\cdot(\xi\restriction k+1)}\subset ((2^N)^N)^{\omega\setminus {k+1}}$ is dense open. Since the orbit of $\xi\restriction [k+1,\omega)$ is dense, we may find some $h\in \mathrm{Sym}(N)^{\omega\setminus k+1}$ so that $h\cdot \xi\restriction[k+1,\omega) \in D_{(\pi,\mathrm{id})\cdot(\xi\restriction k+1)}$. Then $g=(\pi,\mathrm{id},h)$ is the desired extension of $\pi$ so that $g\cdot \xi \in D$.
\end{proof}
It follows now that 
\begin{equation*}
        (\forall^\ast h\in\mathrm{Sym}(N)^\omega)(\forall^\ast g\in \mathrm{Sym}(N)^\omega) g\cdot (h\cdot f(x))) \in C,
\end{equation*}
and so $(\forall^\ast h\in \mathrm{Sym}(N)^\omega)\left[h\cdot f(x) \in C\right]$, as required.
\end{proof}

We now fix sufficiently generic $\alpha$ and $\beta_0$, so that the conclusion of the proposition holds, and so that $f$ is a reduction from $F_n$ to $F^\ast_n$. For any $x\in (2^\N)^\N$, $(\forall^\ast g\in \mathrm{Sym}(N)^n)\left[g\cdot f\circ\psi(x) \in C\right]$. We may now construct the reduction using a large section uniformization theorem. By \cite[18.6]{Kechris-DST-1995} there is a Borel map $\rho\colon X_n \to \mathrm{Sym}(N)^n$ so that $\rho(x)\cdot f\circ\psi(x)\in C$ for any $x\in X_n$. Since $\psi$ and $f$ are reductions, and $F^\ast_n$ is invariant under the action, we conclude that
\begin{equation*}
x\mapsto \rho(x)\cdot f\circ\psi(x)    
\end{equation*}
is a Borel reduction of $F_n$ to $F^\ast_n\restriction C$.

\subsection{Different base sets for $F_n$}\label{section: appendix}

Recall the definition of the equivalence relation $F_n$ on the domain  $X_n\subset ((2^\N)^\N)^n$ in Section~\ref{Section: def of Fn into}. Given a sequence of countable infinite sets $\Vec{N}=\seqq{N_i}{i<n}$, define analogously an equivalence relation $F_n(\Vec{N})$ on a domain $X_n(\Vec{N})\subset\prod_{i<n} (2^{N_i})^{N_{i+1}}$. Fix bijections $\pi_i \colon \N \to N_i$. These lead naturally to a homeomorphism 
\begin{equation*}
    \prod_{i<n} (2^{N_i})^{N_{i+1}} \to \prod _{i<n} (2^\N)^\N = ((2^\N)^\N)^n.
\end{equation*}
Define $X_n(\Vec{N})$ to be the preimage of $X_n$, and $F_n(\Vec{N})$ to be the pullback of $F_n$. 

Equivalently, these objects can be defined directly as in Section~\ref{Section: def of Fn into}. For example, given $x\in \prod_{i<n} (2^{N_i})^{N_{i+1}}$, we may define $A^x_1 = \set{x(0)(t)}{t\in N_1}$, a countable subset of $2^{N_0}$, $A^x_2 = \set{a^{x,s}_1}{s\in N_2}$, where $    a^{x,s}_1 = \set{x(0)(t)}{x(1)(s)(t)=1} \subset A^x_1$, and define $A^x_3, \dots , A^x_n$ analogously. We then have that $x\mathrel{F_n(\Vec{N})}y$ if and only if $A^x_{k+1}= A^y_{k+1}$ for all $k < n$.

\subsection{Proof of Lemma~\ref{lemma: technical reduction to inj and sep}}


We will construct a Borel reduction $\psi\colon F_n\to F_n$ so that for any $x\in \mathrm{Image}(\psi)$, $x(i)\in (2^\N)^\N$ is injective for $i<n$ and separated for $0<i<n$.

First we describe the map in terms of the classifying invariants $A^x_n$ associated to $x\in X_n$. Note that if $A^x_1,\dots,A^x_n$ are all infinite, then there is some $x'\in X_n$ which is injective and is $F_n$-equivalent to $x$ (that is, $A^x_i = A^{x'}_i$ for $1\leq i \leq n$). For an injective $x\in X_n$, the condition of being separated corresponds to: for any $1\leq i < n$, for any $u\neq v \in A_i$ there is some $Z \in A_{i+1}$ so that $Z$ contains exactly one of $\{u,v\}$.
    
    Given some $A_1,\dots,A_n = A^x_1,\dots, A^x_n$ we construct new sets $B_1,\dots,B_n$, which will correspond to  $A^y_1,\dots,A^y_n$ for some $y\in X_n$ which will be defined as $y=\psi(x)$. We may assume that there are infinitely many reals $S = \{\star_1, \star_2, \dots\}$ which are not in $A^x_1$ for any $x$. 
    Define inductively 
    \begin{itemize}
        \item $B_1 = A_1 \cup S$;
        \item $B_{k+1} = A_{k+1} \cup \set{\{a,b\}}{a\in A_k,\, b\in B_k\setminus A_k}$.
    \end{itemize}
Note that $B_i\setminus A_i$ is infinite for $i=1,\dots,n$. For $1\leq i <n$, given $u,v\in B_k$, we find some set $Z\in B_{k+1}$ separating them. If $u,v$ are both in $A_k$, let $a\in B_k\setminus A_k$, then $\{u,a\}\in B_{k+1}$ separates $\{u,v\}$. If $u,v$ are both in $B_k\setminus A_k$, let $a\in A_k$, then $\{u,a\}\in B_{k+1}$ separates $\{u,v\}$. If $u\in A_k$ and $v\in B_k\setminus A_k$, let $a\in B_k\setminus A_k$ be different than $v$, then $\{u,a\}\in B_{k+1}$ separates $\{u,v\}$.

Finally, note that the sets $(A_1,\dots,A_n)$ are uniquely defined from $(B_1,\dots,B_n)$, where the members of $A_i$ are the members of $B_i$ which do not include any member of $S$ in their transitive closure. This fact corresponds to the map we are constructing being a reduction.

\begin{remark}\label{remark: psi joint separation}
    We will later use the following fact. Suppose $x,x'$ are such that $A^x_i = A^{x'}_i$ for $i=1,\dots,k$, $k<n$. Let $B'_i$ be the result of the construction applied to $x'$. Then $B_k = B'_k$, and for any distinct $u,v\in B_k$ there is $Z$ which is both in $B_{k+1}$ and $B'_{k+1}$ so that $Z$ separates $\{u,v\}$.
\end{remark}

It is left to find a Borel map sending $x$ to $y$ so that $A_i^y$ are as $B_i$ to $A_i^x$ as above. Given $x\in X_n$, we define \[\psi_0(x)=y=\seqq{y(k)}{k<n},\] so that $y(0)\in (2^\N)^{\N\sqcup \N}$, and $y(k)\in (2^{\N\sqcup \N^k})^{\N\sqcup \N^{k+1}}$ for $k\geq 1$, as follows. First we define $y(0)$.     For $n\in \N$ (in the left copy of $\N\sqcup \N$), $y(0)(n)=x(0)(n)$. For $n\in \N$ (in the right copy of $\N\sqcup \N$), $y(0)(n)=\star_n$. That is, $y(0)$ comprises of two sequences, one is $x$ and the other is the sequence of new reals $S$ as above. 
    
    Given $y(k)\in (2^{\N\sqcup \N^k})^{\N\sqcup \N^{k+1}}$, define $y(k+1)\in (2^{\N\sqcup \N^{k+1}})^{\N\sqcup \N^{k+2}}$:
    \begin{itemize}
        \item for $n\in \N$, for $m\in \N$,     $y(k+1)(n)(m) = x(n)(m)$, and for $t\in \N^{k+1}$, $y(k+1)(n)(t)=0$.
        \item for $(m,t)\in \N^{k+2} = \N\times \N^{k+1}$, $y(k+1)((m,t))(a)=1$ if and only if $a = m \in \N$ or $a=t\in \N^{k+1}$.
    \end{itemize}
Define $\Vec{N}=\seqq{N_i}{i<n}$ by $N_0 = \N\sqcup \N$, $N_k = \N\sqcup \N^{k+1}$. Then $\psi_0$ is a Borel reduction of $F_n$ to $F_n(\Vec{N})$ so that if $y=\psi_0(x)$ then the sets $A^y_1,A^y_2,\dots$ are constructed from $A^x_1,A^x_2,\dots$ as above. 

Let $\pi$ be the homeomorphism which is a reduction of $F_n(\Vec{N})$ to $F_n$. 
If $y=\pi\circ \psi_0(x)$ then the sets $A^y_1, A^y_2,\dots$ are infinite and separated. The only issue at this point is that $y(k)$ may be a non-injective enumeration of the set $A^y_k$. We invoke a cleanup function eliminating multiplicities, to end up with injective sequences. Recall from Section~\ref{subsection : warmup n=1 case} the Borel map $C\colon (2^\N)^\N\to (2^\N)^\N$ so that for $x\in (2^\N)^\N$,
\begin{itemize}
    \item if the set enumerated by $x$ is finite, $C(x)=x$,
    \item if the set enumerated by $x$ is infinite, $C(x)$ is an injective enumeration of the same set.
\end{itemize}
Another property of this map is that $C(x)$ does not depend on the reals appearing in the sequence $x$, but only on whether $x(n), x(m)$ are equal, for $n,m \in \N$.

Next, extend this to a map $\hat{C}_0\colon (2^\N)^{\color{blue}\N} \times 2^{\color{blue}\N} \to (2^\N)^{\color{blue}\N} \times 2^{\color{blue}\N}$ which, after applying $C$ to the first coordinate, corrects the second coordinate to carve out the same subset. That is, given $(x,\upsilon)\in (2^\N)^\N\times 2^\N$ and $\hat{C}_0(x,\upsilon) = (x',\upsilon')$, $\set{x(i)}{\upsilon(i)=1} = \set{x'(i)}{\upsilon'(i)=1}$. This in turn extends to a map
\begin{equation*}
    \hat{C}\colon (2^\N)^{\color{blue}\N} \times (2^{\color{blue}\N})^\N \to (2^\N)^{\color{blue}\N} \times (2^{\color{blue}\N})^\N,
\end{equation*}
defined so that if $\hat{C}(x,y) = (x',y')$, then $\hat{C}_0(x,y(i)) = (x',y'(i))$, for all $i\in\N$.
For $k+1<n$, define $C_k\colon ((2^\N)^\N)^n\to ((2^\N)^\N)^n$ by
\begin{equation*}
    C_k(x(0),\dots,x(k),x(k+1),x(k+2),\dots) = (x(0), \dots \hat{C}(x(k),x(k+1)), x(k+2),\dots).
\end{equation*}
If $n<\omega$ define $    C_{n-1}(x(0),\dots,x(n-1)) = (x(0),\dots, C(x_{n-1}))$.
Note that the maps $C_k$, $k<n$, commute with one another on the domain $X_n$. Finally, define
\begin{equation*}
C^\ast\colon ((2^\N)^\N)^n\to ((2^\N)^\N)^n\,\textrm{ by }\,  C^\ast = C_0\circ C_1\circ \dots.  
\end{equation*}
Note that $C^\ast$ is well defined also in the case $n=\omega$, as the $n$'th coordinate of $C^\ast(x)$ is fixed by $C_{n+1}\circ C_{n+2}\circ\dots$. The map
\begin{equation*}
    \psi = C^\ast \circ \pi \circ \psi_0 \colon X_n \to X_n
\end{equation*}
is now the desired reduction, concluding the proof of Lemma~\ref{lemma: technical reduction to inj and sep}. \hfill \qedsymbol

\section{Some ideas and obstacles towards the main theorem}\label{Section: obstacles second round}
We begin working towards a proof of Theorem~\ref{Theorem : Main}. First we make a slight reformulation. Then we briefly sketch the ideas for the case $n=1$, and emphasize some difficulties for extending these to $n\geq 2$.

\subsection{A reformulation}\label{subsection : reformulation}
We will prove Theorem~\ref{Theorem : Main} in the following equivalent formulation. For equivalence relations $F$ and $E$ on the same domain, say that $E$ \textbf{extends} $F$ if $F\subset E$. For $n<m\leq \omega$ we may view $F_n$ as an equivalence relation on $X_m$, defined by $x\mathrel{F_n}y\iff x\restriction n\mathrel{F_n} y\restriction n$, for $x,y\in X_m$. In this case $F_n$ extends $F_m$, for $n<m$. 
\begin{thm}\label{Theorem : Main reformulated}
Fix $1\leq n\leq \omega$. For any equivalence relation $E$, classifiable by countable structures, which extends $F_n$, either
    \begin{itemize}
        \item $F_n$ is Borel reducible to $E$, or
        \item For some $k<n$, $E$ extends $F_k$ on a comeager subset of $X_n$.
    \end{itemize}   
\end{thm}

Given a Borel homomorphism $f\colon F\to_B E$, define $E^\ast$, on the same domain as $F$, as the pullback of $E$: $    x\mathrel{E^\ast}y\iff f(x)\mathrel{E}f(y)$. Then (1) $E^\ast$ extends $F$, and (2) $E^\ast$ is Borel reducible to $E$ (witnessed by $f$). The definition above, of $F_n$ as an equivalence relation on $X_m$, is the pullback of $F_n$ by the homomorphism $u^m_n$.

\begin{proof}[Proof of Theorem~\ref{Theorem : Main} from Theorem~\ref{Theorem : Main reformulated}]
Given an analytic equivalence relation $E$ and a Borel homomorphism $f\colon F_n\to_B E$, as in Theorem~\ref{Theorem : Main}, apply Theorem~\ref{Theorem : Main reformulated} to $E^\ast$. In the first case, we conclude that $F_n$ is Borel reducible to $E^\ast$, and is therefore Borel reducible to $E$ as well. Otherwise, there is $k<n$ and a comeager $C\subset X_n$ on which $E^\ast$ extends $F_k$.
We may find a Borel partial function $g^\ast\colon X_k\to X_n$, defined on a comeager subset of $X_k$, so that $g^\ast(x)$ extends $x$, and $g^\ast(x)\in C$, for any $x$ in the domain of $g^\ast$. 
In particular, $g^\ast$ is a homomorphism from $F_k$ to $E^\ast$. Define $g=f\circ g^\ast$. Then $g\colon F_k\to_B E$ is a Borel homomorphism.

Finally, we see that $f$ factors through $u^n_k$, via $g$, on the comeager set $C$. Given $x\in C$, since $E^\ast$ extends $F_k$ on $C$, $x$ and $g^\ast\circ u^n_k(x)$ are $E^\ast$-related, and so $f(x)$ and $f(g^\ast\circ u^n_k(x))=g\circ u^n_k(x)$ are $E$-related. 
\end{proof}

\subsection{The case $n=1$}\label{subsection : warmup dichotomy n=1 case}
Fix an analytic equivalence relation $E$ on $(2^\N)^\N$ so that $F_1\subset E$. Recall that $F_0$ is a trivial equivalence relation, with just one class, so the second clause in Theorem~\ref{Theorem : Main reformulated}, stating that $E$ extends $F_0$ on a comeager set, states that $E$ has a comeager equivalence class.
Assume that $E$ does not have a comeager class. \cite[Theorem~6.24]{ksz} then proves that $F_1$ is Borel reducible to $E$ as follows. Recall the homomorphism $f\colon F_1 \to_B F_1$ defined in Section~\ref{subsection : warmup n=1 case}. We claim that it reduces $F_1$ to $E$. Since $E$ extends $F_1$, it remains to show that $x\not\mathrel{F_1}y \implies f(x)\not\mathrel{E}f(y)$.
Note that there are three different ways for $x,y\in(2^\N)^\N$ to be not $F_1$-related:
\begin{enumerate}[a.]
    \item $x$ and $y$ enumerate disjoint sets;
    \item one of the two sets is contained in the other;
    \item non of the above.
\end{enumerate}
For example, if $x\not\mathrel{F_1}y$ as in case a. above, one can show that then $(f(x),f(y))\in (E^c)^\ast \subset E^c$, and so $f(x)\not\mathrel{E}f(y)$, as required. 

\begin{Notation}\label{notation: union warmup}
For $a,b\in (2^\N)^\N$, write $a\cup b$ for some member of $(2^\N)^\N$ enumeration the union of the sets enumerated by $a$ and $b$. We will ask questions about whether such sequence is equivalent to another, according to $F_1$ or $E$. As both extend $F_1$, the answer does not depend on the enumeration of $a\cup b$. We will similarly use the notations $a\setminus b$ and $a\cap b$, whenever these are not empty, for some member of $(2^\N)^\N$ enumerating the corresponding sets.

Whenever $a\cup b$, $a\cap b$, or $a\setminus b$ are infinite, {\bf we always take the notation to be an injective enumeration} of the corresponding set. One important aspect of the definition of $f$ in Section~\ref{subsection : warmup n=1 case} is that for any $a,b$ in the image of $f$, both are infinite, and the sets $a\cap b$ and $a\setminus b$ are either empty or infinite.
\end{Notation}
Assume now that $x\not\mathrel{F_1} y$ according to case c. Following the definition of $f$, we may write $f(x)$ and $f(y)$ as ${a}\cup {c}$ and ${b}\cup {c}$, where $a=f(x\setminus y)$, $b=f(y\setminus x)$, and $c=f(x\cap y)$.

It follows from $E$ being meager that $\forall^\ast (a,b,c)\in (2^\N)^\N\times (2^\N)^\N\times (2^\N)^\N \left[ a\cup c \not\mathrel{E} b\cup c \right]$. Let $C\subset ((2^\N)^\N)^3$ be the corresponding comeager set. As before, for a ``sufficiently generic'' choice of the function $g\colon 2^\N\to (2^\N)^\N$ in Section~\ref{subsection : warmup n=1 case} it can be verified that for disjoint non-empty $z_1,z_2,z_3$, for almost all $(g_1,g_2,g_3)\in (S_\infty)^3$, $(g_1\cdot f(z_1), g_2\cdot f(z_2), g_3\cdot f(z_3))\in C$. Applying this to the disjoint non-empty sets $x\setminus y, y\setminus x, x\cap y$, it then follows that $f(x)\not\mathrel{E}f(y)$, as required. 

\subsection{The case $n\geq 2$}
Our efforts so far were to find a Borel homomorphism $f\colon F_n \to_B F_n$, landing in comeager sets (after a Vaught transform). Using this homomorphism we will be able to extend the ideas in Section~\ref{subsection : warmup dichotomy n=1 case} to prove the following: if $x\not\mathrel{F_n}y$ differ \emph{only at the last coordinate} (so their restrictions to $X_{n-1}$ are $F_{n-1}$-equivalent), then $f(x)\not\mathrel{E}f(y)$. See Lemma~\ref{lemma: main proof last coordinate reduction}.

A new difficulty arising in the $n\geq 2$ case is dealing with $x\not\mathrel{F_n}y$ which are already $F_k$-inequivalent for some $k<n$. The issue with trying to extend these arguments directly is a proliferation of cases to consider.
For example, suppose $x(0)$ and $y(0)$ are already $F_1$-inequivalent. There are three cases a., b., and c. as above. Consider case c., so we have that $x(0) \cap y(0), x(0)\setminus y(0), y(0)\setminus x(0)$ are not empty. We would now need to view each set in $x(1)$, which is considered a subset of $x(0)$, as a union of two sets, a subset of $x(0)\cap y(0)$ and a subset of $x(0)\setminus y(0)$. Given two different members of $x(1)$, we would need to worry about whether they may agree on either their restrictions to $x(0)\cap y(0)$ or $x(0)\setminus y(0)$. Similarly, we would have to keep track on which members of $x(1)$ and $y(1)$ agree on the intersection $x(0)\cap y(0)$. While this may be handled for $n=2$, when $n>>2$ there are ever more cases to consider and divisions to keep track of.

The solution will be to ``decompose'' the equivalence relation $E$ (which extends $F_n$) to a sequence of equivalence relations $E_k$ so that $E_k$ extends both $F_k$ and $E$. While this is not generally possible, it is possible generically. This is the content of Lemma~\ref{lemma: main proof breakdown of E}.


The proof of Lemma~\ref{lemma: main proof last coordinate reduction} relies on the Baire-categoric techniques developed above, and will be a natural extension of the arguments in Section~\ref{section: complexity on comeager sets}. The proof of Lemma~\ref{lemma: main proof breakdown of E} will involve higher set theoretic techniques as well.

\section{Proof of the main theorem}\label{Section: proof of main thm}

We now prove the main result, Theorem~\ref{Theorem : Main}, in its equivalent formulation, Theorem~\ref{Theorem : Main reformulated}. As above, it will be convenient to work with $F_n^\ast$ instead of $F_n$ (see Section~\ref{section: change of countable set}). We prove the following, which is equivalent to Theorem~\ref{Theorem : Main reformulated}.
\begin{thm}
Fix $1\leq n\leq \omega$. For any $E$ which is classifiable by countable structures and which extends $F^\ast_n$, either
    \begin{itemize}
        \item $F_n$ is Borel reducible to $E$, or
        \item For some $k<n$, $E$ extends $F^\ast_k$ on a comeager subset of $X^\ast_n$.
    \end{itemize}   
\end{thm}

Towards that end, fix $1\leq n \leq \omega$, and $E$ as above which extends $F^\ast_n$. Assume that $E$ does not extend $F^\ast_k$ on a comeager set, for any $k<n$. We must prove that $F_n$ is Borel reducible to $E$.
Recall the definition of $f\colon F_n\to_B F_n$, Definition~\ref{defn: the homomorphism f} in Section~\ref{section: the main construction}, and the definition of $\psi\colon F_n\to F_n$ from Lemma~\ref{lemma: technical reduction to inj and sep}. Since $E$ extends $F^\ast_n$, $f\circ\psi\colon F_n\to_B E$ is a homomorphism.
\begin{claim}
    There are maps $\alpha$ and $\beta_0$ so that $f\circ\psi$ is a reduction of $F_n$ to $E$.
\end{claim}
Towards that end, fix $x,y\in X_n$ in the image of $\psi$ so that $x\not\mathrel{F_n}y$. We need to prove that ${f(x)}\not\mathrel{E}{f(y)}$. We will split into countably many cases. In each case we will show that $f(x)\not\mathrel{E}f(y)$ for generically chosen $\alpha$ and $\beta_0$. This will conclude the proof of the claim.

\begin{lemma}\label{lemma: main proof last coordinate reduction}
    Fix $0\leq k<\omega$. Let $E_{k+1}$ be an analytic equivalence relation, defined on a comeager subset of $X^\ast_{k+1}$, extending $F^\ast_{k+1}$ on this domain. Assume that $E_{k+1}$ does not extend $F^\ast_k$ on any comeager set. Let $x,y\in X_{k+1}$ be in the image of $\psi$. Assume that ${x}\not\mathrel{F_{k+1}}y$ yet ${(x\restriction k)}\mathrel{F_k}{(y\restriction k)}$. Then ${f_{k+1}(x)}\not\mathrel{E_{k+1}}{f_{k+1}(y)}$, where $f_k\colon F_k\to_B F_k$ is the homomorphism $\gamma_0\times \gamma^{k\setminus\{0\}}$ as in Definition~\ref{defn: the homomorphism f}.
\end{lemma}
\begin{lemma}\label{lemma: main proof breakdown of E}
    There are analytic equivalence relations $E_k$ for $k<n$, defined on comeager subsets of $X^\ast_k$, so that
    \begin{enumerate}
        \item $E_k$ extends $F^\ast_k$, on a comeager set;
        \item $E_{k+1}$ does not extend $F^\ast_k$ on any comeager set;
        \item $E \subset E_k$, on a comeager set, for each $k<n$. That is, on a comeager set, if $x\mathrel{E}y$ then $x\restriction k \mathrel{E_k} y\restriction k$.
    \end{enumerate}
In fact we get the following picture, on comeager sets, with $E_n = E$:
\begin{center}

\begin{tikzpicture}[
node/.style={}]
\node[node]      (F1)   at (0,0)        {$F^\ast_1$};
\node[node]      (F2) [right=0.01cm of F1]  {$\supseteq \,\,F^\ast_2$};
\node[node]      (F3) [right=0.01cm of F2]  {$\supseteq \,\,F^\ast_3$};
\node[node]      (conts) [right=0.01cm of F3]  {$\supseteq \,\,\dots$};
\node[node]      (Fn) [right=0.01cm of conts]  {$\supseteq \,\, F^\ast_n$};

\node[node]     (cont1) [below=0.01cm of F1]
{\rotatebox[origin=c]{270}{$\subseteq$}};
\node[node]     (E1) [below=0.01cm of cont1]
{$E_1$};

\node[node]     (cont2) [below=0.01cm of F2]
{\hspace{6mm}\rotatebox[origin=c]{270}{$\subseteq$}};
\node[node]     (E2) [below=0.01cm of cont2]
{$\supseteq \,\, E_2$};

\node[node]     (cont3) [below=0.01cm of F3]
{\hspace{6mm}\rotatebox[origin=c]{270}{$\subseteq$}};
\node[node]     (E3) [below=0.01cm of cont3]
{$\supseteq \,\, E_3$};

\node[node]      (lowconts) [right=0.01cm of E3]  {$\supseteq \,\,\dots$};

\node[node]     (contn) [below=0.01cm of Fn]
{\hspace{6mm}\rotatebox[origin=c]{270}{$\subseteq$}};
\node[node]     (En) [below=0.01cm of contn]
{$\supseteq \,\, E_n$};

\end{tikzpicture}
\end{center}
\end{lemma}
\begin{remark}
The lemma is equivalent to Lemma~\ref{lemma: breaking down homomorphism}. See Section~\ref{subsection : reformulation}.    
\end{remark}

Finally, the proof of the main theorem terminates as follows. Let $k$ be minimal so that $x\restriction k \not\mathrel{F_k} y\restriction k$. It follows from Lemma~\ref{lemma: main proof last coordinate reduction}, and that $f_k(x\restriction k) = f_n(x)\restriction k$, that $f(x) \restriction k \not\mathrel{E_{k}} f(y)\restriction k$, and therefore $f(x) \not\mathrel{E} f(y)$, as required.

We note that, without the assumption that $E$ is classifiable by countable structures, the proof works if we assume that $E$ can be decomposed as in Lemma~\ref{lemma: main proof breakdown of E} above. In the terminology of Theorem~\ref{Theorem : Main}, we get the following variation. 
\begin{thm}
    Given analytic equivalence relations $E_k$ for $k\leq n$ and a diagram of Borel homomorphisms which commute on comeager sets as in Lemma~\ref{lemma: breaking down homomorphism}, so that $f_k$ does not factor through $u^k_l$ for $l<k\leq n$, then ${F_n} \leq_B {E_n}$.
\end{thm}
We conclude the paper by proving Lemma~\ref{lemma: main proof last coordinate reduction} and Lemma~\ref{lemma: main proof breakdown of E}.

\subsection{Proof of Lemma~\ref{lemma: main proof last coordinate reduction}}
For this subsection we fix $f = f_{k+1} \colon F_{k+1} \to F_{k+1}$. Since $f$ is a homomorphism and ${x\restriction k} \mathrel{F_k} {y\restriction k}$, we may assume that $x\restriction k = y\restriction k$. Now $x(k), y(k)\in (2^\N)^\N$ are $F_1$-{\bf in}equivalent. There are three options:
\begin{enumerate}[a.]
    \item The two subsets of $2^\N$ enumerated by $x(k), y(k)$ are disjoint;
    \item one of the two sets is contained in the other;
    \item neither of the above.
\end{enumerate}
When we use set notation, such as $x(k)\cap y(k)$, or $x(k)\setminus y(k)$, we refer to the sets enumerated by $x(k), y(k)$, respectively.
We assume that $k\geq 1$. For $k=0$ the arguments are similar to Section~\ref{subsection : warmup dichotomy n=1 case}.

\subsection*{Case a} Assume that $x(k)$ and $y(k)$ enumerate disjoint subsets of $2^\N$.
Since $x\restriction k = y\restriction k$, and by the definition of $f$, we may write $f(x) = (a,r)$ and $f(y) = (a,s)$, where $a\in ((2^N)^N)^k$ and $r,s\in (2^N)^N$. 
We view $(a,r,s)$ as a member of the space $((2^N)^N)^k\times (2^N)^N\times (2^N)^N$.
The following is a consequence of our assumption that $E_{k+1}$ does not extend $F_k$ on any comeager set.
\begin{lemma}\label{lemma: proof of theorem case a generics inequiv} 
$        \forall^\ast (x,y,z)\in ((2^N)^N)^k\times (2^N)^N\times (2^N)^N\left[(x,y)\not\mathrel{E_{k+1}}(x,z) \right]$.
\end{lemma}
\begin{proof}
    Given $x\in ((2^N)^N)^k$, consider the equivalence relation $E_x$ on $(2^N)^N$, $y\mathrel{E_x}z\iff (x,y)\mathrel{E_{k+1}}(x,z)$. The lemma is equivalent to the statement: for almost every $x\in ((2^N)^N)^k$ every $E_x$ class is meager. If this fails, then for almost every $x\in ((2^N)^N)^k$ there is a comeager equivalence class in $E_x$. In this case, we conclude that $E_{k+1}$ extends $F_k$ on the comeager set of all $(x,y)\in ((2^N)^N)^k\times (2^N)^N$ so that $y$ is in the comeager equivalence class of $E_x$. 
\end{proof}

\subsubsection*{Cosmetic modifications}

As in Section~\ref{subsection : orbit ER presentation}, there is a natural action
\begin{equation*}
\mathrm{Sym}(N)^k\times \mathrm{Sym}(N)\times\mathrm{Sym}(N)    \curvearrowright  ((2^N)^N)^k\times (2^N)^N\times (2^N)^N.
\end{equation*}
More specifically, expressing the group as 
\begin{equation*}
    \mathrm{Sym}(N)^{k-1}\times {\color{blue}\mathrm{Sym}(N)}\times \mathrm{Sym}(N)\times\mathrm{Sym}(N)
\end{equation*}
 and the space as
\begin{equation*}
    ((2^N)^N)^{k-1}\times (2^N)^{\color{blue}N} \times (2^{\color{blue}N})^N \times (2^{\color{blue}N})^N,
\end{equation*}
the blue copy of $\mathrm{Sym}(N)$ acts on the three blue copies of $N$ diagonally, and the two copies of $\mathrm{Sym}(N)$ act separately on the two copies of $(2^{N})^N$, as in Section~\ref{subsection : orbit ER presentation}.
The point is that the two projections $((2^N)^N)^k\times (2^N)^N\times (2^N)^N \to ((2^N)^N)^{k+1}$ (the maps $(a,b,c)\mapsto (a,b)$ and $(a,b,c)\mapsto (a,c)$) are equivariant.

\begin{claim}\label{claim: main proof case a generic permutation works} Given a comeager set $C\subset ((2^N)^N)^k\times (2^N)^N\times (2^N)^N$, for almost any $\beta_0$ and $\alpha$,
\begin{equation*}
    \forall^\ast (g,h_1,h_2)\in \mathrm{Sym}(N)^k\times \mathrm{Sym}(N)\times\mathrm{Sym}(N)\left[(g,h_1,h_2)\cdot(a, r, s) \in C\right].
\end{equation*}
\end{claim}
Applying this claim to the comeager set $C$ we get from Lemma~\ref{lemma: proof of theorem case a generics inequiv}, we conclude that for some $(g,h_1,h_2)$, $(g,h_1)\cdot (a,r) \not\mathrel{E_{k+1}} (g,h_2)\cdot (a,s)$, and so $f(x)\not\mathrel{E_{k+1}}f(y)$, as required for Lemma~\ref{lemma: main proof last coordinate reduction}. We finish Case a. by proving the claim.

\begin{proof}[Proof of the claim]
We will use the following variation of Lemma~\ref{lemma: examples of main construction}, in the case that $k>1$.
\begin{lemma}
For a generic $\alpha$ the following holds. Suppose $x\in (2^\N)^\N$ is separated and $y_1,y_2\in (2^\N)^\N$ are injective and separated. Assume further that $y_1,y_2$ are disjoint. Then the triplet $(\zeta,\xi_1,\xi_2) = (\gamma(x),\gamma(y_1),\gamma(y_2))\in (2^S)^X\times (2^X)^{Y_1} \times (2^X)^{Y_2}$, where $S = X = Y_1 = Y_2 = N$, satisfies the conditions in part (2) of Lemma~\ref{lemma: main technical sufficient condition for genericity}.
\end{lemma}
If $k=1$, the following variation of Lemma~\ref{lemma: example of main construction - combination} will be used:
\begin{lemma}
    For generic $\alpha$ and $\beta_0$ the following holds. Suppose $x\in (2^\N)^\N$ is injective, and $y_1,y_2\in (2^\N)^\N$ are injective and separated. Then the triplet $(\zeta,\xi_1,\xi_2) =  (\gamma_0(x),\gamma(y_1),\gamma(y_2))\in (2^S)^X\times (2^X)^{Y_1} \times (2^X)^{Y_2}$, where $S = X = Y_1 = Y_2 = N$, satisfies the conditions in part (2) of Lemma~\ref{lemma: main technical sufficient condition for genericity}.
\end{lemma}

As in Claim 5.4 
we get: $\forall^\ast (h_0,\dots,h_{k-1}) \in \mathrm{Sym}(N)^k$ the triplet 
    \begin{equation*}
        ((h_{k-2},h_{k-1})\cdot f(x(k-1)), (h_{k-1}, \mathrm{id})\cdot f(x(k)), (h_{k-1}, \mathrm{id})\cdot f(y(k))) \in (2^N)^N \times (2^N)^N \times (2^N)^N
    \end{equation*}
    satisfies the conditions of Lemma~\ref{lemma: main technical sufficient condition for genericity} with $k=2$, $S = X = Y_1 = Y_2 = N$, with respect to the comeager set
    \begin{equation*}
    \set{(a,b,c) \in ((2^N)^N)^3}{((h\restriction k-1)\cdot(f(x)\restriction k-1), a , b, c) \in C}
    \end{equation*}
We conclude from Lemma~\ref{lemma: main technical sufficient condition for genericity} that
\begin{equation*}
     (\forall^\ast (h_0,\dots,h_{k-1})\in\mathrm{Sym}(N)^{k}) (\forall^\ast (h'_{k-1},h^1_k,h^2_k)\in \mathrm{Sym}(N)^3)
\end{equation*}
\begin{equation*}
((h\restriction k-1)\cdot(f(x)\restriction k-1), (h_{k-2},h'_{k-1}\cdot h_{k-1})\cdot f(x(k-1)), (h'_{k-1}\cdot h_{k-1}, h^1_k)\cdot f(x(k)), (h'_{k-1}\cdot h_{k-1}, h^2_k)\cdot f(y(k)) )  
\end{equation*}
is in $C$, and therefore $\forall^\ast (h_0,\dots,h_{k-1},h^1_k,h^2_k)\in\mathrm{Sym}(N)^{k}\times\mathrm{Sym}(N)\times\mathrm{Sym}(N)$
\begin{equation*}
((h\restriction k-1)\cdot(f(x)\restriction k-1), (h_{k-2}, h_{k-1})\cdot f(x(k-1)), (h_{k-1}, h^1_k)\cdot f(x(k)), (h_{k-1}, h^2_k)\cdot f(y(k)))\in C,  
\end{equation*}
concluding the proof of the claim.
\end{proof}

We remark that if $k=0$, the proof is similar, using the following variation of Lemma~\ref{lemma: examples of main construction n=1 case}.
\begin{lemma}        
Suppose $x,y \in (2^\N)^\N$ are injective enumerations of two disjoint subsets of $2^\N$. Then for any comeager $C\subset (2^N)^N\times (2^N)^N$, for a generic $\beta_0$, the pair $(\gamma_0(x),\gamma_0(y))\in (2^N)^N\times (2^N)^N$ satisfies the assumption in Lemma~\ref{lemma: main technical sufficient condition for genericity n=1 case}, with $X=N$, $Y_1 = Y_2 = N$, and $M_1 = M_2 = \N$.
\end{lemma}

\subsection*{Case b} Assume that $x(k)$ enumerates a subsets of $y(k)$. We skip the details of this case, as they are similar and slightly simpler than Case c. 

\subsection*{Case c} Assume that $x(k)$ and $y(k)$ enumerate two sets so that $x(k) \cap y(k)$, $x(k)\setminus y(k)$, and $y(k)\setminus x(k)$ are not empty. As before we focus on the case $k\geq 1$.
We may assume that any member of $2^N$ which appears both in $x(k)$ and $y(k)$ appears in the same coordinate $x(k)(i), y(k)(i)$.
Recall the definition of $\gamma(x(k))\in (2^N)^N$. We may identify it with a member of $(2^N)^{S_1\sqcup S_2}$, where $S_1$ is the set of all $(t,k)\in N$ so that the image of $t$ is contained in $x(k)\cap y(k)$, and $S_2= N\setminus S_1$. Note that both $S_1$ and $S_2$ are infinite. The space $(2^N)^{S_1\sqcup S_2}$ is naturally identified with $(2^N)^{S_1}\times (2^N)^{S_2}$, giving a homeomorphism
\begin{equation*}
 \iota\colon ((2^N)^N)^k \times (2^N)^{S_1} \times (2^N)^{S_2} \to ((2^N)^N)^k \times (2^N)^{N}.
\end{equation*}
Note that we may write $f(x) = \iota(a,\xi_0,\xi_1)$ and $f(y)=\iota(a,\xi_0,\xi_2)$ for some $a\in ((2^N)^N)^k$, $\xi_0\in (2^N)^{S_1}$, $\xi_1,\xi_2\in (2^N)^{S_2}$.
The following is a consequence of our assumption that $E_{k+1}$ does not extend $F_k$ on a comeager set.
\begin{lemma}\label{lemma: proof of theorem case c generics inequiv}
 $\forall^\ast (a,\xi_0,\xi_1,\xi_2)\in ((2^N)^N)^k \times (2^N)^{S_1} \times (2^N)^{S_2} \times (2^N)^{S_2} \left[ \iota(a,\xi_0,\xi_1) \not\mathrel{E_{k+1}} \iota(a,\xi_0,\xi_2) \right]$
\end{lemma}
\begin{proof}
Assume for contradiction that the statement fails. Since almost every $E_{k+1}$ class is dense, and $\iota$ is a homeomorphism, it follows that  \[\forall^\ast (a,\xi_0,\xi_1,\xi_2)\in ((2^N)^N)^k \times (2^N)^{S_1} \times (2^N)^{S_2} \times (2^N)^{S_2} \left[ \iota(a,\xi_0,\xi_1) \mathrel{E_{k+1}} \iota(a,\xi_0,\xi_2) \right]\]
Given $a\in ((2^N)^N)^k$, consider the equivalence relation $E_a$ on $(2^N)^{S_1}\times (2^N)^{S_2}$, $(\xi_0,\xi_1)\mathrel{E_a}(\zeta_0,\zeta_1)\iff \iota(a,\xi_0,\xi_1)\mathrel{E_{k+1}}\iota(a,\zeta_0,\zeta_1)$. Then for almost every $a\in ((2^N)^N)^k$, $\forall^\ast (\xi_0,\xi_1,\xi_2)\in (2^N)^{S_1} \times (2^N)^{S_2} \times (2^N)^{S_2} \left[ (\xi_0,\xi_1) \mathrel{E_{a}} (\xi_0,\xi_2) \right]$. 

Fix a bijection $S_2 \to S_1$, giving a homeomorphism $s\colon (2^N)^{S_1} \to (2^N)^{S_2}$. Then 
\begin{equation*}
    (\forall^\ast a\in ((2^N)^N)^k) (\forall^\ast (\xi_0,\xi_1,\xi_2)\in (2^N)^{S_1} \times (2^N)^{S_2} \times (2^N)^{S_1}) \left[ (\xi_0,\xi_1) \mathrel{E_{a}} (\xi_0, s(\xi_2)) \right].
\end{equation*}
Note that for any $a$ the map $(\zeta_0,\zeta_1)\mapsto (s^{-1}(\zeta_1),s(\zeta_0))$ is an $E_a$-invariant homeomorphism $(2^N)^{S_1}\times (2^N)^{S_2} \to (2^N)^{S_1}\times (2^N)^{S_2}$. 
We conclude that 
\begin{equation*}
    (\forall^\ast a\in ((2^N)^N)^k) (\forall^\ast (\xi_0,\xi_1,\xi_2,\xi_3)\in (2^N)^{S_1} \times (2^N)^{S_2} \times (2^N)^{S_1} \times (2^N)^{S_2})
\end{equation*}
\begin{equation*} 
    (\xi_0,\xi_1)\mathrel{E_a} (\xi_0, s(\xi_2)) \mathrel{E_a} (\xi_2, s(\xi_0)) \mathrel{E_a} (\xi_2,\xi_3)
\end{equation*}
That is, for almost every $a\in ((2^N)^N)^k$ there is a comeager equivalence class for $E_a$. As in Lemma~\ref{lemma: proof of theorem case a generics inequiv} we conclude that $E_{k+1}$ extends $F_k$ on a comeager set, a contradiction. 
\end{proof}

\subsubsection*{Cosmetic modifications}
As in Section~\ref{subsection : orbit ER presentation}, there is a natural action
\begin{equation*}
\mathrm{Sym}(N)^k\times \mathrm{Sym}(S_1)\times\mathrm{Sym}(S_2)\times\mathrm{Sym}(S_2)    \curvearrowright  ((2^N)^N)^k\times (2^N)^{S_1}\times (2^N)^{S_2}\times (2^N)^{S_2}.
\end{equation*}
More specifically, expressing the group as 
\begin{equation*}
    \mathrm{Sym}(N)^{k-1}\times {\color{blue}\mathrm{Sym}(N)}\times \mathrm{Sym}(S_1)\times\mathrm{Sym}(S_2)\times\mathrm{Sym}(S_2)
\end{equation*}
and the space as
\begin{equation*}
    ((2^N)^N)^{k-1}\times (2^N)^{\color{blue}N} \times (2^{\color{blue}N})^{S_1} \times (2^{\color{blue}N})^{S_2}\times (2^{\color{blue}N})^{S_2},
\end{equation*}
the blue copy of $\mathrm{Sym}(N)$ acts on the four blue copies of $N$ diagonally, while the groups $\mathrm{Sym}(S_1)$, $\mathrm{Sym}(S_2)$, $\mathrm{Sym}(S_2)$ act separately on the spaces $(2^N)^{S_1}$, $(2^N)^{S_2}$, $(2^N)^{S_2}$.
The point is that the projection maps $((2^N)^N)^k\times (2^N)^{S_1}\times (2^N)^{S_2}\times (2^N)^{S_2} \to ((2^N)^N)^{k+1}$, $(a,b,c,d) \mapsto \iota(a,b,c)$ and $(a,b,c,d) \mapsto \iota(a,b,d)$, are equivariant.

The following holds for a generic choice of $\beta_0$ and $\alpha$ as in Section~\ref{section: the main construction}.
\begin{claim} Given a comeager set 
\begin{equation*}
C\subset ((2^N)^N)^k \times (2^N)^{S_1} \times (2^N)^{S_2} \times (2^N)^{S_2},    
\end{equation*}
\begin{equation*}
    \forall^\ast (g,\delta_0,\delta_1,\delta_2)\in \mathrm{Sym}(N)^k \times \mathrm{Sym}(S_1) \times \mathrm{Sym}(S_2) \times \mathrm{Sym}(S_2),
\end{equation*}
\begin{equation*}
(g\cdot a,(g(k-1),\delta_0)\cdot \xi_0, (g(k-1),\delta_1)\cdot \xi_1, (g(k-1),\delta_1)\cdot\xi_2) \in C.
\end{equation*}
\end{claim}
Applying this claim to the comeager set $C$ we get from Lemma~\ref{lemma: proof of theorem case c generics inequiv}, we conclude that for some $(g,\delta_0,\delta_1,\delta_2)$, \[\iota((g,\delta_0,\delta_1)\cdot (a,\xi_0,\xi_1)) \not\mathrel{E_{k+1}} \iota((g,\delta_0,\delta_2)\cdot (a,\xi_0,\xi_2)),\] and so $\iota((a,\xi_0,\xi_1)) \not\mathrel{E_{k+1}} \iota((a,\xi_0,\xi_2))$, that is, $f(x)\not\mathrel{E_{k+1}}f(y)$, as required for Lemma~\ref{lemma: main proof last coordinate reduction}.

The proof of the claim is similar to that in Case a., where the following variations of Lemma~\ref{lemma: examples of main construction} and Lemma~\ref{lemma: example of main construction - combination} are used. Recall Remark~\ref{remark: psi joint separation}. Since $x,y$ are in the image of $\psi$, then $x(k), y(k)$ are \textbf{jointly separated}, that is,  for any distinct $n,k\in\N$ there is $i$ in $x(k)\cap y(k)$ so that $x(k)(i)=y(k)(i)$ separates $n,k$.
\begin{lemma}
For generic $\alpha$ and $\beta_0$ the following holds. Fix $x\in (2^\N)^\N$, and $y_1,y_2\in (2^\N)^\N$ so that the sets $y_1\cap y_2$, $y_1\setminus y_2$, $y_2\setminus y_1$, are non-empty. Assume that $y_1,y_2$ are injective and jointly separated.
Let $S_1\subset N$ be the set of all $(t,k)$ for which the image of $t$ is contained in $y_1\cap y_2$, $S_2 = N\setminus S_1$. Let $\xi_1\in (2^N)^{S_1}$, $\xi_2,\xi_3\in (2^N)^{S_2}$ be so that $(\xi_1,\xi_2)$ and $(\xi_1,\xi_3)$ correspond to $\gamma(y_1)$ and $\gamma(y_2)$ via the identification of $(2^N)^{S_1}\times (2^N)^{S_2}$ with $(2^N)^N$. Then the conditions in part (2) of Lemma~\ref{lemma: main technical sufficient condition for genericity} are satisfied for $(\zeta,\xi_1,\xi_2,\xi_3)$, where $k=3$, $S = X = N$, $Y_1 = S_1$, $Y_2 = Y_3 = S_2$, for either
\begin{itemize}
    \item $\zeta = \gamma(x)$, assuming $x$ is separated,
    \item $\zeta = \gamma_0(x)$, assuming $x$ is injective.
\end{itemize}
\end{lemma}
\begin{proof}
The new aspect here is in part (2)(a) of Lemma~\ref{lemma: main technical sufficient condition for genericity}. Fix a finite partial function $\tau\colon N\to \{0,1\}$. For each $i=1,2,3$, we need to find infinitely many $(s,d)\in Y_i$ so that $\xi_i(s,d)(-)$ extends $\tau$. This is analogous to the arguments about $\gamma(x)$ in the proof of Lemma~\ref{lemma: examples of main construction}, just that now we must find $(s,d)\in S_1$, when $i=1$, and $(s,d)\in S_2$, when $i=2,3$.

By assumption, given $t_1\neq t_2\in \N^{<\N}$, there is $i\in y_1\cap y_2$ so that $y_1\circ t_1(i) = y_2 \circ t_1(i) \neq y_2 \circ t_2(i) = y_1\circ t_2(i)$. We can therefore find $s^\ast\in \N^{<\N}$ whose image is in $y_1\cap y_2$, so that for any $s$ which contains $s^\ast$ and any distinct $t_1,t_2$ in the domain of $\tau$, $y_1[s,t_1]\neq y_1[s,t_2]$ and $y_2[s,t_1]\neq y_2[s,t_2]$. (Using the notation from the proof of Lemma~\ref{lemma: examples of main construction}.)

Now for any $d\in\N$, $(s^\ast,d)\in S_1$. As before, for a generic $\alpha$, there are infinitely many $d$ for which $\xi_1(s^\ast,d)(-)=\gamma(y_1)(s^\ast,d)(-)$ extends $\tau$. Next, fix some $s$ extending $s^\ast$ so that the image of $s$ is not contained in $y_1\cap y_2$. Then $(s,d)\in S_2$ for any $d\in\N$. Again we may find infinitely many $d$ for which $\xi_2(s,d)(-)=\gamma(y_1)(s,d)(-)$ extends $\tau$, and infinitely many $d$ for which $\xi_3(s,d)(-)=\gamma(y_2)(s,d)(-)$ extends $\tau$.
\end{proof}

We remark that for $k=0$ the following variation of Lemma~\ref{lemma: examples of main construction n=1 case} is used.
\begin{lemma}        
Suppose $x,y,z \in (2^\N)^\N$ are injective enumerations of pairwise disjoint subsets of $2^\N$. Then for any comeager $C\subset ((2^N)^N)^3$, for a generic $\beta_0$, the triplet $(\gamma_0(x),\gamma_0(y),\gamma_0(z))\in (2^N)^N\times (2^N)^N \times (2^N)^N$ satisfies the assumption in Lemma~\ref{lemma: main technical sufficient condition for genericity n=1 case}, with $X=N$, $Y_1 = Y_2 = Y_3 = N$, and $M_1 = M_2 = M_3 = \N$.
\end{lemma}

\subsection{Proof of Lemma~\ref{lemma: main proof breakdown of E}}

First we recall some background.

\subsubsection{$E$-Pins}\label{subsub: E pins}

Let $E$ be an analytic equivalence relation on a Polish space $X$. Assume that $\P$ is a forcing poset and $\tau$ is a $\P$-name which is forced to be a member of the Polish space $X$, as interpreted in the generic extension. The pair $(\P,\tau)$ is an $\mathbf{E}$\textbf{-pin} if 
\begin{equation*}
    \P\times \P \force \tau_l \mathrel{E} \tau_r,
\end{equation*}
where $\tau_l, \tau_r$ are the interpretation of $\tau$ according to the left and right generics respectively.
\begin{lemma}[see{\cite[Proposition 2.1.2]{Larson-Zapletal-Geometric-2020}}]
For $E$, $\P$, $\tau$ as above, $(\P,\tau)$ is an $E$-pin if and only if in any extension of $V$, given two filters $G_1,G_2$ which are separately $\P$-generic over $V$, $\tau[G_1] \mathrel{E} \tau[G_2]$.
\end{lemma}
The reader is referred to \cite[Chapter~2]{Larson-Zapletal-Geometric-2020} for more on pins. Below we will use specifically-designed pins in certain symmetric ZF models, following~\cite{Sha20}.

\subsubsection{Symmetric models}



Let $\P_n$ be Cohen forcing for producing a generic member of $((2^\N)^\N)^n$. This can be defined as the poset of all Borel sets up to inclusion mod meager (see~\cite{Zapletal-idealized-2008}). An important fact we will use is that given a sufficiently large countable model $M$, the set of $x\in ((2^\N)^\N)^n$ which are $\P_n$-generic over $M$ is comeager.
We will consider the equivalent combinatorial presentation of $\P_n$ as the poset of all finite approximations, ordered by extensions.

Let $G\subset \P_n$ be generic over $V$ and $x\in ((2^\N)^\N)^n$ in $V[G]$ be the corresponding generic member. Note that $x \in X_n$, since $X_n \subset ((2^\N)^\N)^n$ is comeager. Let
\begin{equation*}
    A_k = A^x_k,
\end{equation*}
for $k\leq n$, as in Section~\ref{Section: def of Fn into}. Consider the models $V(A_n)$, the minimal extension of $V$ which contains $A_n$ and satisfies $ZF$. Such models were studied by Monro~\cite{Mon73}. Their relationship to the Friedman-Stanley jumps was introduced in~\cite{Sha20}. 

\begin{remark}
For $k<n$, the poset $\P_n$ can be naturally presented as a product $\P_k \times \P^k_n$, where $\P^k_n$ adds a member of $((2^\N)^\N)^{n\setminus k}$ by finite approximations.
\end{remark}

\begin{fact}\label{fact: An viewed as pin}
    There is a poset $\Q_n$ in $V(A_n)$, definable from $A_n$ over $V$, and a $\Q_n$-name $\sigma_n$ for a member of $X_n$, so that it is forced that $A_n^{\sigma_n} = A_n$. 
\end{fact}
\begin{proof}
Take $\Q_n$ to be the poset to add, by finite approximations, a countable enumerations of the hereditary closure of $A_n$. The sequence $\seqq{A_k}{k<n}$ may be viewed as a member of the space $\prod_{k<n} (2^{A_k})^{A_{k+1}} = (2^\N)^{A_1} \times (2^{A_1})^{A_2} \times (2^{A_2})^{A_3} \times \dots$. After forcing with $\Q_n$, given enumerations of the sets $A_1, A_2, \dots$, we may naturally translate this to a member of $((2^\N)^\N)^n$, as in Section~\ref{section: appendix}, and this will be our $\sigma_n$.    
\end{proof}

\begin{remark}
$(\Q_n,\sigma_n)$ is an $F_n$-pin in the model $V(A_n)$. Given an equivalence relation $E$ which extends $F_n$, then $(\Q_n,\sigma_n)$ is an $E$-pin as well.     
\end{remark}

\begin{defn}
Let $\R^k_n$ to be the product $\Q_k \times \P^k_n$, and let $\rho^k_n$ be an $\R^k_n$-name for the member of $((2^\N)^\N)^n$ whose restriction to $((2^\N)^\N)^k$ is $\sigma_k$, and its restriction to $((2^\N)^\N)^{n\setminus k}$ is added by $\mathbb{P}^k_n$.
\end{defn}
The useful property of $\R^k_n$ is that it allows us to add the set $A_n$ over the model $V(A_k)$ in a sufficiently homogeneous way.
\begin{fact}
There is an $\R^k_n$-generic $R$ over $V(A_k)$ so that $A^{\rho^k_n[R]}_n = A_n$.
\end{fact}
\begin{fact}\label{fact: homogeneity Ak to An factor}
For any two conditions $p,q\in \R^k_n$ there is an automorphism of $\R^k_n$ (in $V(A_k)$) sending $p$ to $q$ and fixing the name for $A^{\rho^k_n}_n$.
\end{fact}

We will use the following property of the models $V(A_n)$ (see~\cite[Lemma~4.5]{Sha20}).
\begin{lemma}\label{lemma: monro models subset of Mk is in Mkplus1}
    If $B\in V(A_n)$ is definable from $A_n$ over $V$, and $B\subset V(A_k)$ for $k<n$, then $B\in V(A_k)$ and is definable from $A_{k}$ over $V$.
\end{lemma}
For example, if $r$ is a real in $V(A_1)$ which is definable from $A_1$, then $r\in V$.
\begin{proof}
Let $\dot{B}$ be the $\R^k_n$-name for the set which is defined from $A^{\rho^k_n}_n$ according to the definition of $B$ from $A_n$. It follows from Remark~\ref{fact: homogeneity Ak to An factor} that for $b\in V(A_k)$, if there is some condition in $\R^k_n$ forcing that $\check{b}\in \dot{B}$, then every condition in $\R^k_n$ forces that $\check{b}\in \dot{B}$. We may now define $B$ in $V(A_k)$ as the set of all $b\in V(A_k)$ for which it is forced that $\check{b}\in\dot{B}$.
\end{proof}

\subsubsection{From pins to sets}
Let $E$ be an equivalence relation which is classifiable by countable structures. 
Using the Scott analysis, we have a complete classification of $E$, $x\mapsto B_x$, assigning a hereditarily countable set to each $x$ in the domain of $E$. This map is absolute, so that $B_x$ is always the same set, no matter in which model (containing $x$) we perform the calculation. 

\begin{lemma}\label{lemma: pin to set}
Let $E$ and $x\mapsto B_x$ be as above and assume that $(\Q_n,\sigma_n)$ is an $E$-pin in $V(A_n)$. Then there is a set $B\in V(A_n)$, definable from $A_n$, so that $\Q_n\force \check{B}=B_{\sigma_n}$.
\end{lemma}
\begin{proof}
For any two $\Q_n$-generics $G_1,G_2$ over $V(A_n)$, $B_{\sigma_n[G_1]} = B_{\sigma_n[G_2]}$. So the set $B = B_{\sigma[G_1]}$ is in $V(A_n)$, definable as the unique set which is forced to be equal to $B_{\sigma_n}$. Note that $B$, just like $A_n$, is likely not hereditarily countable in $V(A_n)$.
\end{proof}

\begin{lemma}\label{lemma: Ek from set in Ak}
Let $B_n\in V(A_n)$ be a set which is definable from $A_n$ over $V$. 
Let $M$ be a sufficiently large countable substructure, $C_n\subset ((2^\N)^\N)^n$ the set of all $\P_n$-generics over $M$. Note that $C_n$ is comeager and $C_n\subset X_n$. Define $E_n$ on $C_n$ by
    \begin{equation*}
        x\mathrel{E_n}y \iff B_n^x = B_n^y,
    \end{equation*}
    where $B^y_n$ is the set defined in $M(A_n^y)$ from $A_n^y$ according to the definition of $B_n$ from $A_n$.
    Then $E_n$ extends $F_n$ on $C_n$. Moreover, for $k<n$, $E_n$ extends $F_k$ on a comeager set if and only if $B_n\in V(A_k)$ is definable from $A_k$ over $V$.
    \end{lemma}
\begin{proof}
If $x\mathrel{F_n}y$ then $A^y_n = A^x_n$ and so $B^x_n = B^y_n$. Therefore $E_n$ extends $F_n$ on $C_n$.

Assume that $B_n\in V(A_k)$ and is definable from $A_k$ over $V$. For $x,y\in C_n$, if $x\restriction k = y\restriction k$ then $M(A^x_k) = M(A^y_k)$ and therefore $B^x_n = B^y_n$, as both are definable using the same definition from $A^x_k$ in the model $M(A^x_k)$, and so $x\mathrel{E_n}y$ by definition. We conclude that $E_n$ extends $F_k$ on $C_n$.

Next, assume that $E_n$ extends $F_k$, $k<n$ on a comeager set. Then for any $\P_n$-generics $x,y$ over $V$, if $A^x_k = A^y_k$ then $B^x_n = B^y_n$, as calculated in $V(A^x_n), V(A^y_n)$ respectively. It follows that for any two $\R^k_n$-generics $R_1,R_2$ over $V(A_k)$, $B_n^{\rho^k_n[R_1]} = B_n^{\rho^k_n[R_2]}$, so $B_n$ can be defined in $V(A_k)$ as the unique set which is forced to be equal to $B_n^{\rho^k_n}$.

\end{proof}

\subsubsection{Concluding Lemma~\ref{lemma: main proof breakdown of E}}

It was convenient to use the equivalence relations $F^\ast_k$ before. For now let us return to the usual presentation and prove the equivalent:
\begin{lemma}
Let $E$ be an equivalence relation, classifiable by countable structures, which extends $F_n$ but does not extend $F_k$ on any comeager set, for $k<n$. Then there are equivalence relations $E_k$ for $k<n$, defined on comeager subsets of $X_k$, so that
    \begin{enumerate}
        \item $E_k$ extends $F_k$ on a comeager set;
        \item $E_{k+1}$ does not extend $F_k$ on any comeager set;
        \item  $E \subset E_k$, on a comeager set, for each $k<n$. That is, on a comeager set, if $x\mathrel{E}y$ then $x\restriction k \mathrel{E_k} y\restriction k$.
    \end{enumerate}
\end{lemma}
Fix $E$ as in the statement of the lemma. Apply Lemma~\ref{lemma: pin to set} to the $E$-pin $(\Q_n,\sigma_n)$ in $V(A_n)$, and get a set $B\in V(A_n)$ as in Lemma~\ref{lemma: pin to set}. $B$ is definable from $(\Q_n,\sigma_n)$ and therefore definable from $A_n$. It follows from Lemma~\ref{lemma: Ek from set in Ak} that     $B\notin V(A_k)$ for $k<n$.

Fix ordinals $\eta$, $\beta$ so that $B\in \mathcal{P}^{\beta}(\eta)$. For $1\leq k < n$ define the ordinal $\alpha_k$ to be the least so that $\mathrm{t.c.}(B) \cap \mathcal{P}^{\alpha_k}(\eta)\notin V(A_{k-1})$, and let $B_k = \mathrm{t.c.}(B) \cap \mathcal{P}^{\alpha_k}(\eta)$. Since $B_k$ is a subset of $V(A_{k})$ which is definable from $A_n$, it follows from Lemma~\ref{lemma: monro models subset of Mk is in Mkplus1} that $B_k\in V(A_k)$ is definable from $A_k$ over $V$. Let $B_n = B$.

We now define equivalence relations $E_k$ from $B_k$, as in Lemma~\ref{lemma: Ek from set in Ak}, so that $E_k$ extends $F_k$ on a comeager subset of $X_k$, and does not extend $F_{k-1}$ on a comeager set.
Note also that for $k<n$, since $B_k$ is defined as the intersection of $B_{k+1}$ with $\mathcal{P}^{\alpha_k}(\eta)$, then $E_{k}$ extends $E_{k+1}$ on a comeager set. Moreover, $E_n$ and $E$ agree on a comeager set. This concludes the proof of the lemma.

\bibliographystyle{alpha}
\bibliography{bibliography}

\end{document}